\documentclass[11pt,reqno]{amsart}
\usepackage[utf8]{inputenc}
\usepackage{amsmath,amssymb,amsthm,amsfonts}
\usepackage{mathrsfs}
\usepackage{amsfonts}
\usepackage{a4wide}

\allowdisplaybreaks \numberwithin{equation}{section}

\usepackage{fancyhdr} 
\allowdisplaybreaks
\numberwithin{equation}{section}

\renewcommand{\headrulewidth}{0pt} 

\usepackage[numbers,sort&compress]{natbib}
\usepackage{color}
\usepackage[colorlinks=true]{hyperref}
\def\alert#1{\textcolor[rgb]{1,0,0}{#1}}
\def\org#1{\textcolor[rgb]{0,1,0}{#1}}
\def\green#1{\textcolor[rgb]{0,1,0}{#1}} 
\def\red{\color{red}}  
\usepackage{hyperref}
\hypersetup{colorlinks=true,
}
\newtheorem{theorem}{Theorem}[section]

\newtheorem{lemma}[theorem]{Lemma}

\newcommand{\nz}{\rm nz}
\newcommand{\C}{{\mathbb C}}
\theoremstyle{definition}

\newtheorem{remark}[theorem]{Remark}

\newcommand{\ep}{\varepsilon}
\newcommand{\Om}{\Omega}
\newcommand{\R}{\mathbb{R}}
\newcommand{\ds}{\displaystyle}

\renewcommand{\baselinestretch}{1}
\newcommand{\Ga}{\Gamma}
\newcommand{\ga}{\gamma}
\newcommand{\ro}{\rho}
\newcommand{\ri}{\Bigarrow}
\renewcommand{\ni}{\noindent}
\newcommand{\pa}{\partial}
\newcommand{\om}{\omega}
\newcommand{\sik}{\sum_{i=1}^k}
\newcommand{\vov}{\Vert\omega\Vert}
\newcommand{\Umy}{U_{\mu_i,y^i}}
\newcommand{\lamns}{\lambda_n^{^{\scriptstyle\sigma}}}
\newcommand{\chiomn}{\chi_{_{\Omega_n}}}
\newcommand{\ullim}{\underline{\lim}}
\newcommand{\bsy}{\boldsymbol}
\newcommand{\mvb}{\mathversion{bold}}
\newcommand{\la}{\lambda}
\newcommand{\La}{\Lambda}
\newcommand{\va}{\varepsilon}
\newcommand{\be}{\beta}
\newcommand{\al}{\alpha}
\newcommand{\vs}{\vspace}
\newcommand{\h}{\hspace}
\newcommand{\na }{\nabla }
\newcommand{\var}{\varphi}
\newcommand{\p}{p(\al,\be)}
\newcommand{\myfontsize}{\fontsize{8pt}{11pt}\selectfont}
\newcommand{\myfontsizes}{\fontsize{8pt}{8pt}\selectfont}
\usepackage{listings}
\usepackage{caption}

\begin{document}
	
	\title
	{Positive solutions of critical Hardy-H\'{e}non equations with  logarithmic term}
	\maketitle
	\begin{center}
		\author{Qihan He}
		\footnote{\myfontsize Email addresses:heqihan277@gxu.edu.cn;College of Mathematics and Information Science, \&
			Center for Applied Mathematics of Guangxi (Guangxi University), Guangxi University, Nanning, Guangxi,  530003, P. R. China
		},
		\author{Wenxuan Liu}
		\footnote{\myfontsize  Email addresses:1923167227@qq.com;College of Mathematics and Information Science, Guangxi University, Nanning,  Guangxi, 530003, P. R. China}
		and
		\author{Yiqing Pan}
		\footnote{\myfontsize  Email addresses:13718049940@163.com;(1).College of Mathematics and Information Science, Guangxi University, Nanning,  Guangxi, 530003, P. R. China;(2).College of Science, Beibu Gulf University, Qinzhou,  Guangxi, 535011, P. R. China}
		
	\end{center}
	
	\begin{abstract} We consider the  existence, non-existence and multiplicity of positive solutions to the following  critical Hardy-H\'enon equation with   logarithmic term
		\begin{equation*}\label{eq11}\left\{
		\begin{array}{ll}
		-\Delta u =|x|^{\alpha}|u|^{2^*_{\alpha}-2}\cdot u+\mu u\log u^2+\lambda u,~~&x\in \Omega,\\
		u=0,~~&x\in \partial \Omega,\\
		\end{array}
		\right.\end{equation*}
		where $ \Omega=B$ for $\alpha\geq 0$,  $ \Omega=B\setminus\{0\}$ for $\alpha\in(-2,0)$,  $B\subset\R^N$ is an unit ball, $\lambda,~\mu \in \R$,~$N\geq 3,~\alpha>-2$, $2^*_{\alpha}:=\frac{2(N+\alpha)}{N-2}$ is the  critical exponent for the embedding $H_{0,r}^{1}( \Omega)\hookrightarrow L^p( \Omega;|x|^\alpha)$, and which can be seen as a Br\'{e}zis-Nirenberg problem. When $N \geq 4$ and $\mu>0$, we will show that the above problem has a positive Mountain pass solution, which is also a ground state solution. At the same time, when $\mu<0$,  under some assumptions on the $N$, $\mu$, $\lambda$ and $\alpha$, we will show that the above problem has at least a positive least energy solution and at least  a positive Mountain pass solution, respectively.  What's more, when certain inequality related to $N \geq 3$, $\mu<0 $ and $\alpha\in(-2,0]$ holds, we will demonstrate the non-existence of  positive solutions to the above-mentioned problem. The presence of logarithmic term brings some new and interesting phenomena to this problem.
	\end{abstract}
	\section{Introduction and main results}
	In the paper, we study   the existence,  non-existence and multiplicity  of positive solutions for the following equation
	\begin{equation}\label{eq11}\left\{
	\begin{array}{ll}
	-\Delta u =|x|^{\alpha}|u|^{2^*_{\alpha}-2}\cdot u+\mu u\log u^2+\lambda u,~~&x\in  \Omega,\\
	u=0,~~&x\in \partial  \Omega,\\
	\end{array}
	\right.\end{equation}
	where $ \Omega=B$ for $\alpha\geq 0$,  $ \Omega=B\setminus\{0\}$ for $\alpha\in(-2,0)$, $B\subset\R^N$ is an unit ball, $\lambda, ~\mu \in \R$, $N\geq 3,~\alpha>-2$, and  $ 2^*_{\alpha}:=\frac{2(N+\alpha)}{N-2}$ is the  critical exponent for the embedding $
	H_{0,r}^{1}( \Omega)\hookrightarrow L^p( \Omega;|x|^\alpha).$
	
	Our motivation for studying \eqref{eq11} comes from  that it originates from  some variational problems in geometry and physics(see \cite{Yamabe,Aubin,bre1,tau1,tau2,uhl,BahriA}),  where lack of compactness  occurs.
	The most notorious example is Yamabe's problem: find a function $u$ satisfying
	\begin{equation*}
	\begin{cases}
	-4\frac{N-1}{N-2}\Delta u=R'{\left|u\right|}^{{2}^{\ast }-2}u-R(x)u,& \text{ }x\in{M },\\
	\quad \;\:\, u>0,&\text{ } x\in{M },
	\end{cases}
	\end{equation*}
	where $R'$ is some constant, $M$ is an $N$-dimensional Riemannian manifold, $\Delta$ denotes the Laplacian and $R(x)$ represents the scalar curvature.
	
	When $\alpha=\lambda=\mu=0$, the equation \eqref{eq11} turns into the following equation:
	\begin{equation}\label{715}
	\begin{cases}
	-\Delta u={\left|u\right|}^{{2}^{\ast }-2}u,& \text{ }x\in{\Omega },\\
	\quad \;\:\, u=0,&\text{ } x\in{\partial \Omega }.
	\end{cases}
	\end{equation}
	As we know, the solvability of this problem \eqref{715} depends heavily on the geometry and topology of the domain $\Omega$.  Poho\v{z}aev \cite{po} showed the first result to problem \eqref{715}: If the bounded domain $\Omega$ is star-shaped, then problem \eqref{715} has no non-trivial solution. Some other results about \eqref{715} can be seen in \cite{Ba,co,ka,pa,pa1} and the references therein. In \cite{bre2}, Br\'{e}zis and Nirenberg alleged that a lower-order term, such as $\lambda u$, also can reverse this circumstance. Br\'{e}zis and Nirenberg considered the following classical problem
	\begin{equation}\label{fc1.3}
	\begin{cases}
	-\Delta u={\left|u\right|}^{{2}^{\ast }-2}u+\lambda u,& \text{ }x\in{\Omega },\\
	\quad \;\:\, u=0,&\text{ } x\in{\partial \Omega },
	\end{cases}
	\end{equation}
	and asserted that the existence of a solution depends heavily on the parameter  $\lambda$ and the dimension $N$. They  showed  that:
	$(i)$ when $N \ge 4$ and  $\lambda\in \left(0,\lambda_{1}(\Omega)\right)$, there exists a positive solution for  \eqref{fc1.3};
	$(ii)$ when $N=3$ and $\Omega$ is a ball, problem  \eqref{fc1.3} has a positive solution if and only if  $\lambda\in \left(\frac{1}{4}\lambda_{1}(\Omega),\lambda_{1}(\Omega)\right)$;
	$(iii)$~problem  \eqref{fc1.3} has no solution when $\lambda<0$ and $\Omega$ is star-shaped, where $\lambda_1(\Omega)$ denotes the first eigenvalue of $-\Delta $ with zero Dirichlet boundary value.
	In \cite{Wangc},
	Wang and Su  considered the following equation,
	\begin{equation}\label{eq17}
	\left\{%
	\begin{array}{ll}
	-\Delta u=|x|^\alpha u^{2^*_{\alpha}-1},~~&\hbox{in} ~\R^N,\\
	u\in D_r^{1,2}(\R^N),u>0~&\hbox{in}~\R^N,
	\end{array}%
	\right.\end{equation}
	where $\alpha>0$.
	It follows from \cite{lieb,gida,glad} that for any $\alpha>-2$, equation \eqref{eq17} has a unique radial solution given by
	\begin{equation}
	u_{\varepsilon,\alpha}(x)=\frac{C_{\alpha,N}\varepsilon^{\frac{N-2}{2}}}{(\varepsilon^{2+\alpha}+|x|^{2+\alpha})^{\frac{N-2}{2+\alpha}}},
	\end{equation}
	where $\varepsilon>0$ and $C_{\alpha,N}={[(N+\alpha)(N-2)]}^\frac{N-2}{2(2+\alpha)}$.
	The functions $u_{\varepsilon,\alpha}(x)$ are the extremal functions for the following inequality:
	\begin{equation}\label{1.9}
	\int_{\R^N}|\nabla u|^2\geq S_\alpha\Big(\int_{\R^N}|x|^\alpha|u|^{2^*_{\alpha}}\Big)^{\frac{2}{2^*_{\alpha}}},u\in D_r^{1,2}(\R^N).
	\end{equation}
	As we know, the presence of the term $|x|^\alpha$ in equation \eqref{eq17} drastically changes the problem. For this kind of nonlinearities, it is not possible to apply the moving plane method anymore to get the radial symmetry around some point, and indeed non-radial solutions appear. This phenomenon has brought attention to the H\'{e}non problem, i.e. \eqref{eq17} or
	the H\'{e}non equation (see \cite{hen}) with Dirichlet boundary conditions:
	\begin{equation}\label{eq1.2}
	\left\{%
	\begin{array}{ll}
	-\Delta u=|x|^\alpha u^{p-1},~~&\hbox{in} ~\Omega,\\
	u>0, ~&\hbox{in}~\Omega,\\
	u=0,~&\hbox{on}~\partial \Omega.
	\end{array}%
	\right.\end{equation}
	In \cite{smets}, Smets et al. are interested in the symmetry of ground states solutions of \eqref{eq1.2}. It is easy to verify that
	\begin{equation}\label{eq1.3}
	S_{\alpha,p}(\Omega):=\inf\limits_{u\in H_0^1(\Omega)\backslash{\{0\}}}\frac{\int_\Omega|\nabla u|^2}{(\int_\Omega|x|^\alpha |u|^p)^\frac{2}{p}}
	\end{equation}
	is achieved at least by a positive function, where $\alpha \ge 0$, and $p$ satisfies:\:$(i)$ $N=2,~ 2<p<2^*=+\infty$;\:$(ii)$  $N\ge3, ~2<p<2^*=\frac{2N}{N-2}$.
	Moreover, since \eqref{eq1.2} and \eqref{eq1.3} are invariant under rotations of $\Omega$, they consider naturally
	\begin{equation}\label{eq1.4}
	S^R_{\alpha,p}(\Omega):=\inf\limits_{u\in H_{0,rad}^1(\Omega)\backslash{\{0\}}}\frac{\int_\Omega|\nabla u|^2}{(\int_\Omega|x|^\alpha |u|^p)^\frac{2}{p}},
	\end{equation}
	where $H_{0,rad}^1(\Omega)$ denotes the space of radial functions in $H_0^1(\Omega)$. It is also easy to verify that $S^R_{\alpha,p}(\Omega)$ is achieved by a positive function $v$. Under the variational methods to the H\'{e}non equation \eqref{eq1.2}, they showed that  there was $\alpha^*>0$
	such that for $\alpha>\alpha^*$ and any(see \cite{pi}) $2 < p < 2^*:=\frac{2N}{N-2}$, the ground state solution of \eqref{eq1.2} was non-radial.
	In \cite{ni}, Ni showed that \eqref{eq1.2} had a radial solution for $2 < p < 2^*_{\alpha}$.
	
	In \cite{ph}, Phan and Souplet considered the following problem:
	\begin{equation}\label{717}
	-\Delta u=|x|^\alpha u^{p-1}~~x\in\Omega,
	\end{equation}
	where $\alpha\in\R$, $p>1$ and $\Omega$ is a domain of $\R^N$ with $N\geq 2$. They are concerned in particular with the Liouville property, i.e., the non-existence of positive solutions in the whole space $\R^N$. If $\alpha\leq -2$, then \eqref{717} has no positive solution  in any domain $\Omega$ containing the origin(see \cite{gida,Bi,Du}). In \cite{Bi1,gida}, the authors proved that $(i)$ $N\geq 2$, $\alpha>-2$ and $1<p<2^*_\alpha$, then \eqref{717} has no positive radial solution in $\Omega=\R^N$; $(ii)$ $N\geq 2$, $\alpha>-2$ and $p\geq2^*_\alpha$, then \eqref{717} possesses bounded, positive radial solution in $\Omega=\R^N$. Phan et al. proved that if $N\geq 2$, $\alpha>-2$ and $p\geq2^*_\alpha$, there exists a bounded sequence of real numbers
	$b_k>0$ and a sequence of solutions $u_k$ of \eqref{717} with $u(x)=\varphi(x)$ in $\partial \Omega$, such that $u_k(0)\rightarrow\infty$ as $k\rightarrow\infty$, where $\Omega=B$ and $b_k\equiv\varphi_k$. In \cite{gh}, Ghoussoub and Yuan proved that when $\Omega$ is a star-shaped domain in $\R^N$ and $p=2^*_\alpha$, \eqref{717} with $u(x)=0$ in $\partial \Omega$ has no non-trivial solution and in \cite{ni}, Ni proved that \eqref{717} with $u(x)=0$ in $\partial \Omega$  has no solution for $p\geq 2^*_\alpha$.
	In \cite{kds}, Kang and Peng considered the following singular problem:
	\begin{equation}\label{eq7101}
	\left\{%
	\begin{array}{ll}
	-\Delta u-\mu\frac{u}{|x|^2}={|x|^\alpha}{|u|^{2^*_{\alpha}-2}}u+\lambda|u|^{q-2}u,~~ &x\in\Omega,\\
	u=0,~~&x\in\partial\Omega,
	\end{array}%
	\right.\end{equation}
	where $-2< \alpha\leq 0$, $0\leq \mu<\bar{\mu}:=(\frac{N-2}{2})^2$, and $q\in(2,2^*)$. They pointed out that the problem \eqref{eq7101} has been studied by some authors (see \cite{jane,cdm,ferr,garc}), many interesting results have been obtained when $\alpha=0$. Especially, in \cite{jane}, Jannelli  proved the existence of positive solutions for problem \eqref{eq7101} as $\alpha=0$ and $q=2$. In \cite{kds}, Kang and Peng  are interested in  whether the above results remain true for \eqref{eq7101} as $-2<\alpha<0$, with the critical Sobolev-Hardy growth? They discussed the existence of positive solutions for the singular problem \eqref{eq7101} via variational methods. They have shown that there is a positive solution in $H_0^1(\Omega)$ for problem \eqref{eq7101} under the conditions of $-2< \alpha\leq0$, $0\leq\mu<\bar{\mu}$, $\lambda>0$ and some suitable assumption of $q$. Under the hypotheses of $-2<\alpha\leq0$, $q=2$, and $\beta=2(\sqrt{\bar{\mu}}+\sqrt{\bar{\mu}-\mu})$, they asserted that the problem\eqref{eq7101} has a positive solution in $H_0^1(\Omega)$, provided that one of the following conditions hold:
	
	$(i)$ $0<\lambda<\lambda_1(\mu)$ and $0\leq \mu\leq \bar{\mu}-1$;
	
	$(ii)$  $\bar{\mu}-1<\mu<\bar{\mu}$ and $\lambda_*(\mu)<\lambda<\lambda_1(\mu)$, where
	\begin{equation*}
	\lambda_*(\mu)=\min\limits_{\varphi\in H_0^1(\Omega)\setminus\{0\}}\frac{\int_\Omega(|\nabla\varphi|^2/|x|^\beta)}{\int_\Omega(\varphi^2/|x|^\beta)}.
	\end{equation*}
	So, when $\mu=0$ and $q=2$, the equation \eqref{eq7101} turns into the following equation:
	\begin{equation}
	\left\{%
	\begin{array}{ll}
	-\Delta u={|x|^\alpha}{|u|^{2^*_{\alpha}-2}}u+\lambda u,~~ &x\in\Omega,\\
	u=0,~~&x\in\partial\Omega,
	\end{array}%
	\right.\end{equation}
	and then we can obtain that equation \eqref{eq7101} has a positive solution.
	%
	
	When $\alpha=0$, \eqref{eq11} can be written as the following problem  with critical exponent and logarithmic perturbation
	\begin{equation}\label{Dend1.1}
	\begin{cases}
	-\Delta u={\left|u\right|}^{{2}^{\ast }-2}u+\lambda u+\theta u\log {u}^{2}, & \text{ }x\in{\Omega },\\
	\quad \;\:\,  u=0,&\text{ } x\in{\partial \Omega },
	\end{cases}
	\end{equation}
	where the logarithmic term  $u\log {u}^{2}$ brings us two interesting  facts: one is that
	compared with $|u|^{2^*-2}u$, $u\log u^2$ is also  a lower-order term at infinity, and the other one is that $u\log {u}^{2}$ is sub-linear growth near $0$ and sign-changing in $(0, +\infty)$.
	In \cite{Deng}, Deng et al. investigated the existence and non-existence of positive solutions to \eqref{Dend1.1} and  proved that problem \eqref{Dend1.1} admits a positive ground state solution, which is also a Mountain pass type solution when $N \ge4, ~\lambda\in\R$ and $\theta>0$. However,
	that the Nehari manifold is a natural manifold can not be proved. So they  can not find a positive ground state solution to \eqref{Dend1.1} for the case of $\theta<0$. The readers can refer to \cite{Deng} for more details. Later,  Zou and his collaborators gave an important method to find a positive ground state solution for some problems, whose Nehari manifold is not a natural manifold. Zou and his collaborators \cite{ha,ha1} considered the following problem
	\begin{equation}\label{Dend1.2}
	\begin{cases}
	-\Delta u=\mu{\left|u\right|}^{{2}^{\ast }-2}u+\lambda u+\theta u\log {u}^{2}, & \text{ }x\in{\Omega },\\
	\quad \;\:\,  u=0,&\text{ } x\in{\partial \Omega },
	\end{cases}
	\end{equation}
	which is more general than \eqref{Dend1.1} and showed the following results:
	$(i)$ If  $N \geq 4$  and  $(\lambda, \mu, \theta) \in M_{1} \cup M_{2}$, then problem \eqref{Dend1.2} has a positive local minimum solution  $\bar{u}$  and a positive ground state solution  $\tilde{u}$  such that  $J(\bar{u})=\tilde{C}_{\rho}<0$  and  $J(\tilde{u})=\tilde{C}_{\kappa}<0$; $(ii)$ If  $N \geq 4$, $(\lambda, \mu, \theta) \in M_{1} \cup M_{2}$  and  $|\theta| e^{\frac{N}{2}-\frac{\lambda}{\theta}-1}<\rho^{2}$, then  $\tilde{C}_{\rho}=\tilde{C}_{\kappa}$, which implies that the solution  $\bar{u}$  is also a positive ground state solution, where
	\begin{align}
	\begin{array}{cccccc}
	M_{1}:=\{(\lambda, \mu, \theta)|\lambda \in[0, \lambda_{1}(\Omega)), \mu>0, \theta<0, \frac{\mu}{N}(\frac{\lambda_{1}(\Omega)-\lambda}{\mu \lambda_{1}(\Omega)})^{\frac{N}{2}} S^{\frac{N}{2}}+\frac{\theta}{2}|\Omega| >0\}, \\
	M_{2}:=\{(\lambda, \mu, \theta)|\lambda \in \R, \mu>0, \theta<0, \frac{1}{N} \frac{1}{\mu^{\frac{N-2}{2}}} S^{\frac{N}{2}}+\frac{\theta}{2} e^{-\frac{\lambda}{\theta}}|\Omega| >0\}, \\
	J(u)=\frac{1}{2} \int_{\Omega}|\nabla u|^{2}-\frac{\mu}{2^{*}} \int|u_{+}|^{2^{*}}-\frac{\theta}{2} \int(u_{+})^{2}(\log (u_{+})^{2}+\frac{\lambda}{\theta}-1), \\
	\tilde{C}_{\kappa}:=\inf \limits_{u \in\{v \in H_{0}^{1}(\Omega): J^{\prime}(v)=0\}} J(u), \\
	\tilde{C}_{\rho}:=\inf\limits_{|\nabla u|_{2}<\rho} J(u).
	\end{array}
	\end{align}
	and  $\rho$  is a suitable constant. Recently, He and Pan \cite{heq} showed that the local minimum solution  $\bar{u}$  and the  positive ground state solution  $\tilde{u}$, found in \cite{ha,ha1}, have the same functional energy value, i.e., $\tilde{C}_{\rho}=\tilde{C}_{\kappa}$, without the restrictions that  $|\theta| e^{\frac{N}{2}-\frac{\lambda}{\theta}-1}<\rho^{2}$. At the same time, they proved that if $(\lambda, \mu, \theta) \in M_{1} \cup M_{2}$ and $N=3, 4, 5$, then \eqref{Dend1.2} has another positive solutions $u_2$, which is in fact a Mountain pass solution.

	Inspired by the above results, particularly \cite{Deng, ha, ha1, heq},  we want to study the existence, non-existence and multiplicity of positive solutions to \eqref{eq11}.
	
	To find positive solutions to equation \eqref{eq11}, we define a energy  functional:
	\begin{equation}
	I(u)=\frac{1}{2}\int_\Omega|\nabla u|^2-\frac{1}{2^*_{\alpha}}\int_\Omega|x|^\alpha|u_+|^{2^*_{\alpha}}-\frac{\lambda}{2}\int_\Omega u_+^2-\frac{\mu}{2}\int_\Omega u_+^2(\log u_+^2-1),u\in H_{0,r}^1(\Omega),
	\end{equation}
	which can also be written as
	\begin{equation}
	I(u)=\frac{1}{2}\int_\Omega|\nabla u|^2-\frac{1}{2^*_{\alpha}}\int_\Omega|x|^\alpha{|u_+|}^{2^*_{\alpha}}-\frac{\mu}{2}\int_\Omega u_+^2(\log u_+^2+\frac{\lambda}{\mu}-1),u\in H_{0,r}^1(\Omega),
	\end{equation}
	where ${u_+}=\max\{u,0\}$ and ${u_-}=-\max\{-u,0\}$. It is easy to see that $I$ is well-defined in $H_{0,r}^1(\Omega)$, and any non-negative critical point of $I$ corresponds to a solution of equation \eqref{eq11}.
	
	Before stating our results, we  introduce  some notations here.  Unless otherwise stated, we let  $\int$ denote $\int_\Omega$, and let $\lambda_1(\Omega)$ and $S_\alpha$ be the first eigenvalue of $-\Delta$ with zero Dirichlet boundary value, and the best constant of the embedding $H_{0,r}^{1}(\Omega)\hookrightarrow L^{2^*_\alpha}(|x|^{\alpha};\Omega)$, respectively, i.e.,
	\begin{equation}
	\lambda_1(\Omega):=\inf\limits_{u\in H_{0,r}^1(\Omega)\backslash{\{0\}}}\frac{\int_\Omega|\nabla u|^2}{\int_\Omega|u|^2}
	\end{equation}
	and
	\begin{equation}\label{710}
	S_\alpha=\inf\limits_{u\in D_r^{1,2}(\R^N)}\frac{\int_{\R^N}|\nabla u|^2} {(\int_{\R^N}|x|^\alpha|u|^{2^*_{\alpha}})^{\frac{2}{2^*_{\alpha}}}}.
	\end{equation}
	We let
	\begin{equation}\label{7103}
	u_{\varepsilon,\alpha}(x)=\frac{C_{\alpha,N}\varepsilon^{\frac{N-2}{2}}}{(\varepsilon^{2+\alpha}+|x|^{2+\alpha})^{\frac{N-2}{2+\alpha}}},
	\end{equation}
	where  $\varepsilon>0$ and $C_{\alpha,N}={[(N+\alpha)(N-2)]}^\frac{N-2}{2(2+\alpha)}$.
	It follows from \cite{kds,Wangc} that,  up to  dilations, $u_{\varepsilon,\alpha}$ are the extremal functions for $S_{\alpha}$ and positive solutions to \eqref{eq17} with $\alpha>-2$.
	
	Setting
	\begin{equation*}
	\|v\|=(\int_\Omega|\nabla v|^2)^\frac{1}{2},~v\in H_{0,r}^1(\Omega),\quad \mathcal{N}:=\{u\in H_{0,r}^1(\Omega)\backslash\{0\}|g(u)=0\},
	\end{equation*}
	and
	\begin{equation}
	c_g:=\inf\limits_{u\in\mathcal{N}}I(u),\quad c_M:=\inf\limits_{\gamma\in\Gamma}\max\limits_{t\in[0,1]}I(\gamma(t)),
	\end{equation}
	where
	\begin{equation*}
	g(u)=\int|\nabla u|^2-\int|x|^\alpha|u_+|^{2^*_{\alpha}}-\lambda\int u_+^2-\mu\int u_+^2\log u_+^2,
	\end{equation*}
	\begin{equation}
	\Gamma:=\{\gamma\in C([0,1],H_{0,r}^1(\Omega))|\gamma(0)=\bar{u},I(\gamma(1))<I(\bar{u})\},
	\end{equation}
	and $\bar{u}$ will be given in Theorem \ref{Th61}.
	
	Let
	\begin{equation}
	\begin{split}
	&A_0:=\{(\lambda,\mu)|\lambda\in\R,~\mu>0\} ,\\
	&B_0:=\Big\{(\lambda,\mu)|\lambda\in[0,\lambda_1(\Omega)),~\mu<0,~\frac{\alpha+2}{2(N+\alpha)}\Big(\frac{\lambda_1(\Omega)-\lambda}
	{\lambda_1(\Omega)}\Big)^{\frac{N+\alpha}{\alpha+2}}S_\alpha^{\frac{N+\alpha}{\alpha+2}}+\frac{\mu}{2}|\Omega|>0\Big\},\\
	&C_ 0:=\Big\{(\lambda,\mu)|\lambda\in\R,~\mu<0,~\frac{\alpha+2}{2(N+\alpha)}S_\alpha^{\frac{N+\alpha}{2+\alpha}}+\frac{\mu}{2}e^{-\frac{\lambda}{\mu}}|\Omega|>0\Big\},
	\end{split}
	\end{equation}
	and
	\begin{equation}\label{cdef1}
	\tilde{C}_{\rho}:=\inf\limits_{|\nabla u|_2<\rho}I(u),~~\quad~~ \tilde{C}_{\kappa}:=\inf\limits_{u\in\kappa}I(u),
	\end{equation}
	where $\rho>0$ will be given in Lemma \ref{lemma22} and
	\begin{equation}\label{kappadef}
	\kappa:=\{u\in H_{0,r}^1(\Omega):I'(u)=0\}.
	\end{equation}
	It is easy to check that $\tilde{C}_{\rho}$ and $\tilde{C}_{\kappa}$ are well-defined.
	
	After introducing the notations, we proceed to present the main results as follows.
	\begin{theorem}\label{Th41}
		If $(\lambda,\mu)\in A_0$ and $N\geq 4$, then \eqref{eq11} has a positive Mountain pass solution, which is also a ground state solution.
	\end{theorem}
	~~Denote $h(s):=|x|^{\alpha}|s|^{2^*_{\alpha}-2}s+\lambda s+\mu s\log s^2$. It is easy to see that $\mathcal{N}\neq\varnothing$ and $c_M\geq c_g$, if problem \eqref{eq11} has a positive Mountain pass solution. On the other hand, when $\lambda\in\R$ and $\mu>0$, $\frac{h(s)}{s}$ is strictly increasing in $(0,+\infty)$ and strictly decreasing in $(-\infty,0)$, which enable one to show that $c_M\leq c_g$(see \cite[Theorem 4.2]{will}). Therefore, the ground state energy $c_g$ equals to the Mountain pass level energy $c_M$, which
	implies that the Mountain pass solution must be a ground state solution. So, in Theorem \ref{Th41}, we only need to
	show that problem \eqref{eq11} has a positive Mountain pass solution.

	\begin{theorem}\label{Th61}
		~Assume that $N \geq 3$ and $(\lambda,\mu)\in B_0\cup C_0$.
		Then \eqref{eq11} has at least  a positive least energy solution $\bar{u}$ such that $|\nabla \bar{u}|_2<\rho$ and  $I(\bar{u})=\tilde{C}_{\rho}=\tilde{C}_{\kappa}<0$,
		where $\rho>0$ will be given in Lemma \ref{lemma22}, and $\tilde{C}_{\rho}$ and $\tilde{C}_{\kappa}$ are defined in \eqref{cdef1}.
	\end{theorem}

	\begin{theorem}\label{Th64}
		If $3\leq N<\min\{6,6+2\alpha\}$
		and $(\lambda,\mu)\in B_0\cup C_0$,
		then \eqref{eq11} has a positive Mountain pass solution.
	\end{theorem}
	
	\begin{remark}
		We can not find a positive Mountain pass solution to \eqref{eq11} for $N\geq 6$ and $(\lambda,\mu)\in B_0\cup C_0$ in Theorem \ref{Th64}, since $\int|x|^\alpha U_{\varepsilon,\alpha}^{2^*_\alpha-1}=O(\varepsilon^{\frac{N-2}{2}})\leq O(\varepsilon^2)=\int U_{\varepsilon,\alpha}^2$ for $N\geq 6$(see the proof of Lemma \ref{lemma56}). Therefore, we have to leave it as an open problem and continue to study it in the future.
	\end{remark}

	\begin{remark}
		In Theorems \ref{Th61} and \ref{Th64}, we require $(\lambda,\mu)\in B_0\cup C_0$, just to ensure that the functional $I(u)$ has a Mountain pass geometry structure. We don't know whether the domains of the parameters are optimal. If one can find some other domains to make the functional $I(u)$ has the Mountain pass geometry structure, then $(\lambda,\mu)\in B_0\cup C_0$ can be replaced by the corresponding domains.
	\end{remark}
	

	\begin{theorem}\label{Th16}
		Assume that $N\geq3$ and $\alpha\in(-2,0]$. If $\mu<0$ and $-\frac{(N-2)\mu}{\alpha+2}+\frac{(N-2)\mu}{\alpha+2}\log\left(-\frac{(N-2)\mu}{\alpha+2}\right)+\lambda-\lambda_1{(\Omega)}\geq0$, then problem \eqref{eq11} has no positive solution.
	\end{theorem}

	Before closing our introduction, let's briefly introduce our idea of the  proofs.

	We will prove Theorem \ref{Th41} by using the Mountain pass Theorem and estimating the Mountain pass level energy.

	For Theorem   \ref{Th61},  we firstly show that $\tilde{C}_{\rho}$ and $\tilde{C}_{\kappa}$ can be attained by $\bar{u}$ with $|\nabla \bar{u}|_2<\rho$ and $\tilde{u}\in \kappa,$ respectively. Then, by the definitions of $\tilde{C}_{\rho}$ and $\tilde{C}_{\kappa}$, we  have that $\bar{u}\in\kappa$ and
	\begin{align}\label{eq117}
	\tilde{C}_{\rho}=I(\bar{u})\geq\tilde{C}_{\kappa}.
	\end{align}
	On the other hand, since $\tilde{C}_{\kappa}$ can be attained by $\tilde{u}\in \kappa,$ there exists at least a positive solution $\tilde{u}$ of \eqref{eq11} such that $I(\tilde{u})=\tilde{C}_{\kappa}<0$. Letting $g(t)=I(t\tilde{u}),~ t\in\R^{+}$, direct computations imply that $g(t)$  has at most two extreme points, and the geometry structure of the energy functional $I$ tells us that $g(t)$  has at least two extreme points. So we obtain that $g(t)$ has only two extreme points $t_1,~t_2\in\R^+$ with $t_1<t_2$. Using the geometry structure of the energy functional $I$ and the facts that $\tilde{u}$ is a positive solution of \eqref{eq11} and $I(\tilde{u})<0$ again, we can see that $t_1=1$ is the local minimum point of $g(t)$ and $\|t\tilde{u}\|<\rho$ for any $t\in(0,1]$, which implies that $\|\tilde{u}\|<\rho$. That is, $\tilde{u}\in A:=\{|\nabla u|_2<\rho\}$. It therefore follows from $\tilde{C}_{\kappa}=I(\tilde{u})\geq\tilde{C}_{\rho}$ and \eqref{eq117} that $\tilde{C}_{\kappa}=\tilde{C}_{\rho}$.
	
	To find a positive Mountain pass solution to \eqref{eq11}, we firstly prove some important inequalities(see  Lemma \ref{lemma54}).~Based on these inequalities, we construct a path $\gamma_0(t):=\bar{u}+tTU_{\varepsilon,\alpha}$ such that $c_M\leq \sup\limits_{t\in[0,1]}I(\gamma_0(t))<\tilde{C}_{\kappa}+\frac{\alpha+2}{2(N+\alpha)}S_{\alpha}^{\frac{N+\alpha}{2+\alpha}}$, where $U_{\varepsilon,\alpha}$ will be defined in \eqref{eq31} and $T$ will be established in Lemma \ref{lemma45}.  After getting $c_M<\tilde{C}_{\kappa}+\frac{\alpha+2}{2(N+\alpha)}S_{\alpha}^{\frac{N+\alpha}{2+\alpha}}$, it is not difficult to show that any $(PS)_{c_M}$ sequence $\{v_n\}$ of $I$ is convergent in $H_{0,r}^1(\Omega)$. Therefore, we get the Theorem \ref{Th64}.

	For the non-existence of positive solutions, we argue by contradiction. Assuming that $u_0$ is a positive solution of \eqref{eq11} and testing \eqref{eq11} by the first eigenfunction $\varphi_1(x)$ of $-\Delta$ with zero Dirichlet condition, then we have that $\int_\Omega(|x|^\alpha u_0^{2_\alpha^*-2}+\lambda-\lambda_1(\Omega)+\mu\log u_0^2)u_0\varphi_1(x)=0.$ On the other hand, based on the assumptions, we show that   $\int_\Omega(|x|^\alpha u_0^{2_\alpha^*-2}+\lambda-\lambda_1(\Omega)+\mu\log u_0^2)u_0\varphi_1(x)\geq \int_\Omega( u_0^{2_\alpha^*-2}+\lambda-\lambda_1(\Omega)+\mu\log u_0^2)u_0\varphi_1(x)>0.$ Therefore, we can get a contradiction and deduce that
	\eqref{eq11} has no positive solutions.

	The paper is organized as follows: In Section 2, we will check the Mountain pass geometry structure for $I(u)$, under different specific situations, and  give some other preliminaries. We are devoted to estimate the Mountain pass level $c_M$ for different parameters
	$\lambda$, $N$, $\mu$ with $\mu>0$ and show the Theorem \ref{Th41} in Section 3. We put the proof of Theorem \ref{Th61} into the Section 4. Some important estimates and the proof of Theorem   \ref{Th64} will given  in Section 5. Section 6  is contributed to the proof of Theorem \ref{Th16}.

	\section{Preliminaries}
	\begin{lemma}\label{lemma21}{\rm{(see \cite{ni,su1,su2})}} Assume that $V$ satisfies $V(r)\in C(0,R],~V(r)\geq0,~V(r)\not\equiv0$ and there exists $\theta\geq-2$ such that
		\begin{equation}
		\limsup\limits_{r\rightarrow0^+}\frac{V(r)}{r^\theta}<\infty.
		\end{equation}
		Then  the embedding $H_{0,r}^1(B_R)\hookrightarrow L^p(B_R;V)$ is continuous for $1\leq p\leq 2^*_{\theta}$, and  the embedding $H_{0,r}^1(B_R)\hookrightarrow L^p(B_R;V)$ is compact for $1\leq p< 2^*_{\theta}$ if  $\theta> -2$.
	\end{lemma}
	\begin{lemma}\label{lemma22} Assume that $N\geq3$ and $(\lambda,\mu)\in A_0\cup B_0\cup C_0$. Then the functional $I(u)$ satisfies the Mountain pass geometry structure:\\
		(i)~there exist $\sigma,~\rho>0$ such that $I(v)\geq\sigma$ for all $\|v\|=\rho$;\\
		(ii)~there exists $\omega\in H_{0,r}^1(\Omega)$ such that $\|\omega\|\geq\rho$ and $I(\omega)<0$.
	\end{lemma}
	\begin{proof}
		We   divide the proof  into three cases.
		
		\textbf{Case I :}
		$(\lambda,\mu)\in A_0$
		
		Since $\mu>0$, by  $s^2\log s^2\leq Cs^{2^*}$ for all $s\in [1,+\infty)$, we can get that
		\begin{align*}
		\mu\int u_+^2(\log u_+^2+\frac{\lambda}{\mu}-1)
		&\leq \mu\int u_+^2\log(e^\frac{\lambda}{\mu}u_+^2)\\
		&\leq \mu\int_{\{e^\frac{\lambda}{\mu}u_+^2\geq1\}}u_+^2\log(e^\frac{\lambda}{\mu}u_+^2)\\
		&
		\leq C\mu\int_{\{e^\frac{\lambda}{\mu}u_+^2\geq1\}}e^\frac{(2^*-2)\lambda}{2\mu}|u_+|^{2^*}\\
		&\leq Ce^\frac{(2^*-2)\lambda}{2\mu}\mu\int|u_+|^{2^*}\\
		&\leq Ce^\frac{(2^*-2)\lambda}{2\mu}\mu||u||^{2^*}.
		\end{align*}
		So
		\begin{equation}
		I(u)\geq\frac{1}{2}\|u\|^2-C_1\|u\|^{2^*_{\alpha}}-C_2\|u\|^{2^*},
		\end{equation}
		where $C_1,~C_2$ are two positive constants and which implies that there exist $\sigma>0$ and $\rho>0$ such that $I(v)\geq\sigma>0$ for all $\|v\|=\rho$.
		
		Let $0\leq\varphi\in H_{0,r}^1(\Omega)\backslash\{0\}$ be a fixed function, then
		\begin{align*}
		I(t\varphi)&=\frac{t^2}{2}\int|\nabla \varphi|^2-\frac{t^{2^*_{\alpha}}}{2^*_{\alpha}}\int|x|^{\alpha}|\varphi|^{2^*_{\alpha}}-\frac{\mu}{2}t^2\int\varphi^2(\log(t^2\varphi^2)+\frac{\lambda}{\mu}-1)\\
		&=\frac{t^2}{2}\int|\nabla \varphi|^2-\frac{t^{2^*_{\alpha}}}{2^*_{\alpha}}\int|x|^{\alpha}|\varphi|^{2^*_{\alpha}}-\frac{\mu}{2}t^2\log t^2\int\varphi^2-\frac{\mu}{2}t^2\int\varphi^2(\log\varphi^2+\frac{\lambda}{\mu}-1)\\
		&\to -\infty,~~\hbox{as}~t\rightarrow+\infty.
		\end{align*}
		Therefore, we can choose $t_0\in\R^+$ large enough such that
		\begin{align}
		I(t_0\varphi) <0\quad\hbox{and}\quad\|t_0\varphi\|>\rho.
		\end{align}
		So there is a function $\omega\in H_{0,r}^1(\Omega)$ such that $||\omega||\geq\rho$ and $I(\omega)<0$.

		\textbf{Case II:}~~ $(\lambda,\mu)\in B_0$

		For
		$\mu<0$,  direct computations imply that
		\begin{equation}\label{eq62}
		-\frac{\mu}{2}\int u_{+}^{2}(\log u_{+}^{2}-1)\geq \frac{\mu}{2}|\Omega|.
		\end{equation}
		Therefore,
		\begin{equation}\label{eq72}
		\begin{split}
		I(u)&=\frac{1}{2}\int|\nabla u|^2-\frac{1}{2^*_{\alpha}}\int|x|^\alpha|u_+|^{2^*_{\alpha}}-\frac{\lambda}{2}\int|u_+|^2-\frac{\mu}{2}\int u_+^2(\log u_+^2-1)\\
		&\geq\frac{1}{2}\int|\nabla u|^2\Big(1-\lambda\frac{\int |u_+|^2}{\int|\nabla u|^2}\Big)-\frac{1}{2^*_{\alpha}}\frac{\int|x|^\alpha|u_+|^{2^*_{\alpha}}}{(\int|\nabla u|^2)^{\frac{2^*_{\alpha}}{2}}}\cdot\Big(\int|\nabla u|^2\Big)^{\frac{2^*_{\alpha}}{2}}+\frac{\mu}{2}|\Omega|\\
		&\geq\frac{1}{2}\int|\nabla u|^2\Big(1-\frac{\lambda}{\lambda_1(\Omega)}\Big)-\frac{1}{2^*_{\alpha}}\Big[\frac{(\int|x|^\alpha|u_+|^{2^*_{\alpha}})^{\frac{2}{2^*_{\alpha}}}}{\int|\nabla u|^2}\Big]^{\frac{2^*_{\alpha}}{2}}\cdot\Big(\int|\nabla u|^2\Big)^{\frac{2^*_{\alpha}}{2}}+\frac{\mu}{2}|\Omega|\\
		&\geq\frac{1}{2}\frac{\lambda_1{(\Omega)}-\lambda}{\lambda_1(\Omega)}\int|\nabla u|^2-\frac{1}{2^*_{\alpha}}S_\alpha^{-\frac{2^*_{\alpha}}{2}}\Big(\int|\nabla u|^2\Big)^{\frac{2^*_{\alpha}}{2}}+\frac{\mu}{2}|\Omega|.
		\end{split}
		\end{equation}
		We let $\sigma=\frac{\alpha+2}{2(N+\alpha)}(\frac{\lambda_1{(\Omega)}-\lambda}{\lambda_1(\Omega)})^\frac{N+\alpha}{\alpha+2}
		S_\alpha^\frac{N+\alpha}{\alpha+2}+\frac{\mu}{2}|\Omega|$ and $\rho=(\frac{\lambda_1{(\Omega)}-\lambda}{\lambda_1(\Omega)})^\frac{N-2}{2(\alpha+2)}S_\alpha^\frac{N+\alpha}{2(\alpha+2)}$. Then it follows from $(\lambda,\mu)\in B_0$ that $\sigma>0$ and $\rho>0$. Following from  \eqref{eq72},
		\begin{align*}
		I(v)&\geq\frac{1}{2}\frac{\lambda_1{(\Omega)}-\lambda}{\lambda_1(\Omega)}\rho^2-\frac{1}{2^*_{\alpha}}
		S_\alpha^{-\frac{2^*_{\alpha}}{2}}\rho^{2^*_{\alpha}}+\frac{\mu}{2}|\Omega|\\
		&=\frac{\alpha+2}{2(N+\alpha)}\Big(\frac{\lambda_1{(\Omega)}-\lambda}
		{\lambda_1(\Omega)}\Big)^\frac{N+\alpha}{\alpha+2}S_\alpha^\frac{N+\alpha}{\alpha+2}+\frac{\mu}{2}|\Omega|\\
		&=\sigma>0,
		\end{align*}
		for any $\|v\|=\rho$. Similar to Case I, we can find a function $\omega\in H_{0,r}^1(\Omega)$ such that $\|\omega\|\geq\rho$ and $I(\omega)<0$.
		
		\textbf{Case III:}~~ $(\lambda,\mu)\in C_0$

		By
		$\mu<0$ and direct computations,  we have
		\begin{equation}\label{eq82}
		\begin{split}
		-\frac{\lambda}{2}\int u_{+}^{2}-\frac{\mu}{2}\int u_{+}^{2}(\log u_{+}^{2}-1)
		& =-\frac{\mu}{2}\int u_{+}^{2}\log (e^{\frac{\lambda}{\mu}-1}u_{+}^{2})\\
		&\geq -\frac{\mu}{2}\int_{\{e^{\frac{\lambda}{\mu}-1}u_+^2\leq1\}} u_{+}^{2}\log (e^{\frac{\lambda}{\mu}-1}u_{+}^{2})\\ &\geq \frac{\mu}{2}e^{-{\frac{\lambda}{\mu}}}|\Omega|.
		\end{split}
		\end{equation}
		Therefore,
		\begin{equation}\label{eq92}
		\begin{split}
		I(u)
		&\geq\frac{1}{2}\int|\nabla u|^2-\frac{1}{2^*_{\alpha}}\frac{\int|x|^\alpha|u_+|^{2^*_{\alpha}}}{(\int|\nabla u|^2)^{\frac{2^*_{\alpha}}{2}}}\cdot\Big(\int|\nabla u|^2\Big)^{\frac{2^*_{\alpha}}{2}}+\frac{\mu}{2}e^{-\frac{\lambda}{\mu}}|\Omega|\\
		&\geq\frac{1}{2}\int|\nabla u|^2-\frac{1}{2^*_{\alpha}}\Big[\frac{(\int|x|^\alpha|u_+|^{2^*_{\alpha}})^{\frac{2}{2^*_{\alpha}}}}{\int|\nabla u|^2}\Big]^{\frac{2^*_{\alpha}}{2}}\cdot\Big(\int|\nabla u|^2\Big)^{\frac{2^*_{\alpha}}{2}}+\frac{\mu}{2}e^{-\frac{\lambda}{\mu}}|\Omega|\\
		&\geq\frac{1}{2}\int|\nabla u|^2-\frac{1}{2^*_{\alpha}}S_\alpha^{-\frac{2^*_{\alpha}}{2}}\Big(\int|\nabla u|^2\Big)^{\frac{2^*_{\alpha}}{2}}+\frac{\mu}{2}e^{-\frac{\lambda}{\mu}}|\Omega|.
		\end{split}
		\end{equation}
		Letting $\sigma=\frac{\alpha+2}{2(N+\alpha)}S_\alpha^\frac{N+\alpha}{\alpha+2}+\frac{\mu}{2}e^{-\frac{\lambda}{\mu}}|\Omega|$ and  $\rho=S_\alpha^\frac{N+\alpha}{2(\alpha+2)}$, then $\sigma,\rho >0$. Following from  $(\lambda,\mu)\in C_0$ and  \eqref{eq92}, we have that
		\begin{align*}
		I(v)\geq\frac{1}{2}\rho^2-\frac{1}{2^*_{\alpha}}S_\alpha^{-\frac{2^*_{\alpha}}{2}}\rho^{2^*_{\alpha}}+\frac{\mu}{2}e^{-\frac{\lambda}{\mu}}|\Omega|
		=\frac{\alpha+2}{2(N+\alpha)}S_\alpha^\frac{N+\alpha}{\alpha+2}+\frac{\mu}{2}e^{-\frac{\lambda}{\mu}}|\Omega|=\sigma>0,
		\end{align*}
		for any $\|v\|=\rho$. Similar to  Case I, we can  also find a function $\omega\in H_{0,r}^1(\Omega)$ such that $\|\omega\|\geq\rho$ and $I(\omega)<0$.
	\end{proof}

	\begin{lemma}
		\label{lemma23}~Assume that $N\geq3$,~ $\lambda\in\R$ and $\mu\in\R\backslash\{0\}$.~Then any $(PS)_c$ sequence $\{u_n\}$ of $I$ must be bounded in $H_{0,r}^1(\Omega)$ for all $c\in\R.$
	\end{lemma}
	\begin{proof}
		According to the definition of $(PS)_c$,  we have that, as $n\rightarrow\infty$,
		\begin{align}
		I(u_n) \rightarrow c\quad\hbox{and}\quad I'(u_n)\rightarrow0\quad\hbox{in}~ H^{-1}(\Omega).
		\end{align}
		Therefore,
		\begin{equation}\label{eq29}
		\begin{split}
		I(u_n)&=\frac{1}{2}\int|\nabla u_n|^2-\frac{1}{2^*_{\alpha}}\int|x|^\alpha|(u_n)_+|^{2^*_{\alpha}}-\frac{\lambda}{2}\int|(u_n)_+|^2-\frac{\mu}{2}\int (u_n)_+^2(\log (u_n)_+^2-1)\\
		&=c+o_n(1),\\
		\end{split}
		\end{equation}
		and
		\begin{equation}\label{eq210}
		\int|\nabla u_n|^2-\int|x|^\alpha|(u_n)_+|^{2^*_{\alpha}}-\lambda\int|(u_n)_+|^2-\mu\int (u_n)_+^2\log (u_n)_+^2
		=o_n(1)\|u_n\|.
		\end{equation}
		
		We   divide the proof  into two cases.
		
		\textbf{Case I:} ~~~
		$\mu>0$

		It follows from \eqref{eq29} and \eqref{eq210} that
		\begin{equation}\label{eq112}
		\begin{split}
		c+o_n(1)+o_n(1)\|u_n\|&=I(u_n)-\frac{1}{2}\langle I'(u_n),u_n \rangle\\
		&=\frac{\alpha+2}{2(N+\alpha)}\int|x|^\alpha|(u_n)_+|^{2^*_{\alpha}}+\frac{\mu}{2}\int|(u_n)_+|^2\\
		&\geq\frac{\mu}{2}\int|(u_n)_+|^2.
		\end{split}
		\end{equation}
		So we obtain
		\begin{equation}\label{bu2.12}
		C+C\|u_n\|\geq|(u_n)_+|_2^2.
		\end{equation}
		Using  \eqref{eq29} and \eqref{eq210} again,  we have that, for $n$ large enough,
		\begin{align*}
		&2c+\|u_n\|\\
		&\geq I(u_n)-\frac{1}{2^*_{\alpha}}\langle I'(u_n),u_n \rangle\\
		&=\frac{\alpha+2}{2(N+\alpha)}\int|\nabla u|^2-\frac{\lambda(\alpha+2)}{2(N+\alpha)}\int (u_n)_+^2+\frac{\mu}{2}\int (u_n)_+^2-\frac{\mu(\alpha+2)}{2(N+\alpha)}\int (u_n)_+^2\log (u_n)_+^2,
		\end{align*}
		which, together with the following inequality(see \cite{shu} or \cite[Theorem 8.14]{lieb1})
		\begin{equation}
		\int u^2\log u^2\leq\frac{a}{\pi}\|u\|^2+(\log |u|_2^2-N(1+\log a))|u|_2^2
		\quad\hbox{for~any}~u\in H_0^1(\Omega)~
		\hbox{and~any}~a>0,
		\end{equation}
		implies that
		\begin{equation}
		2c+\|u_n\|+\frac{\lambda(\alpha+2)}{2(N+\alpha)}\int (u_n)_+^2-\frac{\mu}{2}\int (u_n)_+^2+\frac{\mu(\alpha+2)}{2(N+\alpha)}\int (u_n)_+^2\log (u_n)_+^2\geq\frac{\alpha+2}{2(N+\alpha)}\int|\nabla u|^2.
		\end{equation}
		According to the facts that $H_{0}^1(\Omega)\hookrightarrow L^p(\Omega),~~1\leq p\leq2^*$ and $|s^2\log s^2|\leq Cs^{2-\delta}+Cs^{2+\delta}$ with $\delta \in(0,1)$, we obtain
		\begin{equation*}
		\begin{split}
		\frac{\alpha+2}{2(N+\alpha)}\|u_{n}\|^{2} & \leq 2 c+\|u_{n}\|+C \int(u_{n})_{+}^{2}+\frac{\mu(\alpha+2)}{2(N+\alpha)}\int\left(u_{n}\right)_{+}^{2} \log \left(u_{n}\right)_{+}^{2} \\
		& \leq 2c+C\left\|u_{n}\right\|+\frac{\mu(\alpha+2)}{2(N+\alpha)}\left[\frac{a}{\pi}\left\|u_{n}\right\|^{2}+\left(\log \left|\left(u_{n}\right)_{+}\right|_{2}^{2}-N(1+\log a)\right)\left|\left(u_{n}\right)_{+}\right|_{2}^{2}\right] \\
		& \leq 2c+C\left\|u_{n}\right\|+\frac{\alpha+2}{4(N+\alpha)}\left\|u_{n}\right\|^{2}+\left||\left(u_{n}\right)_{+}|_{2} ^{2} \log |\left(u_{n}\right)_{+}|_{2}^{2}\right|+C\left|\left(u_{n}\right)_{+}\right|_{2} ^{2} \\
		& \leq 2c+C\left\|u_{n}\right\|+\frac{\alpha+2}{4(N+\alpha)}\left\|u_{n}\right\|^{2}+C\left|\left(u_{n}\right)_{+}\right|_{2}^{2-\delta}+C\left|\left(u_{n}\right)_{+}\right|_{2}^{2+\delta}+C\left|\left(u_{n}\right)_{+}\right|_{2}^{2} \\
		& \leq 2c+C\left\|u_{n}\right\|+\frac{\alpha+2}{4(N+\alpha)}\left\|u_{n}\right\|^{2}+C\left(C+C\left\|u_{n}\right\|\right)^{\frac{2-\delta}{2}}+C\left(C+C\left\|u_{n}\right\|\right)^{\frac{2+\delta}{2}},
		\end{split}
		\end{equation*}
		where  $a>0$  with  $\frac{a}{\pi} \mu<\frac{1}{2}$ and the last inequality has used \eqref{bu2.12} and which implies that  there exists  $C>0$  such that  $\left\|u_{n}\right\|<C$.
		
		\textbf{Case II:}~~$\mu<0$

		If $n$ is large enough, then we have
		\begin{align*}
		2c+\|u_n\|&\geq\frac{\alpha+2}{2(N+\alpha)}\int|\nabla u|^2-\frac{\lambda(\alpha+2)}{2(N+\alpha)}\int (u_n)_+^2+\frac{\mu}{2}\int (u_n)_+^2-\frac{\mu(\alpha+2)}{2(N+\alpha)}\int (u_n)_+^2\log (u_n)_+^2\\
		&=\frac{\alpha+2}{2(N+\alpha)}\|u_n\|^2-\frac{\mu(\alpha+2)}{2(N+\alpha)}\Big(\frac{\lambda}{\mu}\int (u_n)_+^2-\frac{N+\alpha}{\alpha+2}\int (u_n)_+^2+\int (u_n)_+^2\log (u_n)_+^2\Big)\\
		&=\frac{\alpha+2}{2(N+\alpha)}\|u_n\|^2-\frac{\mu(\alpha+2)}{2(N+\alpha)}\Big(\int(u_n)_+^2\log \Big(e^{\frac{\lambda}{\mu}-\frac{N+\alpha}{\alpha+2}}(u_n)_+^2\Big)\Big)\\
		&\geq\frac{\alpha+2}{2(N+\alpha)}\|u_n\|^2-\frac{\mu(\alpha+2)}{2(N+\alpha)}\int_{\{e^{\frac{\lambda}{\mu}-\frac{N+\alpha}{\alpha+2}}(u_n)_+^2\leq1\}}(u_n)_+^2\log \Big(e^{\frac{\lambda}{\mu}-\frac{N+\alpha}{\alpha+2}}(u_n)_+^2\Big)\\
		&\geq\frac{\alpha+2}{2(N+\alpha)}\|u_n\|^2-\frac{\mu(\alpha+2)}{2(N+\alpha)}\int_{\{e^{\frac{\lambda}{\mu}-\frac{N+\alpha}{\alpha+2}}(u_n)_+^2\leq1\}}-e^{\frac{N+\alpha}{\alpha+2}-\frac{\lambda}{\mu}-1}\\
		&\geq\frac{\alpha+2}{2(N+\alpha)}\|u_n\|^2+\frac{\mu(\alpha+2)}{2(N+\alpha)}e^{\frac{N+\alpha}{\alpha+2}-\frac{\lambda}{\mu}-1}|\Omega|,
		\end{align*}
		which implies that $\{u_n\}$ is bounded in $H_{0,r}^1(\Omega)$.
	\end{proof}

	\begin{lemma}\label{lemma51}
		Assume that $N\geq3$ and $\mu<0$. Then $-\infty<\tilde{C}_{\rho}<0$, where $\tilde{C}_{\rho}$ is given in \eqref{cdef1}.
	\end{lemma}
	\begin{proof}
		For any $u\in A:=\{|\nabla u|_2<\rho\}$, by \eqref{eq82}, we have that
		\begin{align*}
		I(u)&=\frac{1}{2}\int|\nabla u|^2-\frac{1}{2^*_{\alpha}}\int|x|^\alpha{|u_+|}^{2^*_{\alpha}}-\frac{\mu}{2}\int u_+^2(\log u_+^2+\frac{\lambda}{\mu}-1)\\
		&\geq-C_1\|u\|^{2^*_{\alpha}}-C_2|\Omega|\\
		&>-C_1\rho^{2^*_{\alpha}}-C_2|\Omega|,
		\end{align*}
		where $C_1,~C_2>0$ and which tells us that $\tilde{C}_{\rho}>-\infty$.
		
		Letting  $0\leq\varphi\in H_{0,r}^1(\Omega)\backslash\{0\}$ be a fixed function, then we have that, for $t>0$ small enough,
		\begin{align*}
		I(t\varphi)&=\frac{t^2}{2}\int|\nabla \varphi|^2-\frac{t^{2^*_{\alpha}}}{2^*_{\alpha}}\int|x|^{\alpha}|\varphi|^{2^*_{\alpha}}-\frac{\mu}{2}t^2\int\varphi^2(\log(t^2\varphi^2)+\frac{\lambda}{\mu}-1)\\
		&=t^2\big(\frac{1}{2}\int|\nabla \varphi|^2-\frac{t^{2^*_{\alpha}-2}}{2^*_{\alpha}}\int|x|^{\alpha}|\varphi|^{2^*_{\alpha}}-\frac{\mu}{2}\log t^2\int\varphi^2-\frac{\mu}{2}\int\varphi^2(\log\varphi^2+\frac{\lambda}{\mu}-1)\big)\\
		&<0.
		\end{align*}
		Therefore, we can choose $t_0\in\R^+$ small enough such that
		\begin{align}
		I(t_0\varphi) <0\quad\hbox{and}\quad|\nabla (t_0\varphi)|_2<\rho,
		\end{align}
		which implies that $\tilde{C}_{\rho}\leq I(t_0\varphi)<0$.
		We obtain
		$-\infty<\tilde{C}_{\rho}<0.$
	\end{proof}
	\begin{lemma}\label{lemma52}
		Assume that $\tilde{C}_{\rho}$ can be obtained by $\bar{u}$. If  $N\geq3$ and $\mu<0$, then $-\infty<\tilde{C}_{\kappa}<0$,
		where $\tilde{C}_{\kappa}$ is given in \eqref{cdef1}.
	\end{lemma}
	\begin{proof}
		According to the assumptions, we have that  $I'(\bar{u})=0,~I(\bar{u})=\tilde{C}_{\rho}<0$. So $\tilde{C}_{\kappa}\leq I(\bar{u})=\tilde{C}_{\rho}<0$.
		For any $u\in\kappa$, where $\kappa$ is given in \eqref{kappadef}, we have that
		\begin{align*}
		I(u)&= I(u)-\frac{1}{2^*_{\alpha}}\langle I'(u),u \rangle
		=\frac{\alpha+2}{2(N+\alpha)}\int|\nabla u|^2-\frac{\mu(\alpha+2)}{2(N+\alpha)}\int u^2\log\left(e^{\frac{\lambda}{\mu}-\frac{N+\alpha}{\alpha+2}}u^2\right)\\
		&\geq -\frac{\mu(\alpha+2)}{2(N+\alpha)}\int u^2\log\left(e^{\frac{\lambda}{\mu}-\frac{N+\alpha}{\alpha+2}}u^2\right)
		\geq \frac{\mu(\alpha+2)}{2(N+\alpha)}e^{\frac{N+\alpha}{\alpha+2}-\frac{\lambda}{\mu}-1}|\Omega|
		>-\infty.
		\end{align*}
		Therefore, we have $-\infty<\tilde{C}_{\kappa}<0$.
	\end{proof}
	\begin{lemma}\label{lemma53}
		Assume that $I'(u)=0$ with $u>0$, and set $g(t):=I(tu),~t>0$. Then $g(t)$ has at most two extreme points.
	\end{lemma}
	\begin{proof}
		By direct computations, we have
		\begin{align*}
		g'(t)&=t\Big(\int|\nabla u|^2-t^{2^*_{\alpha}-2}\int|x|^{\alpha}|u|^{2^*_{\alpha}}-\lambda\int u^2-\mu\int u^2\log t^2-\mu\int u^2\log u^2\Big)\\
		&=t\Big((1-t^{2^*_{\alpha}-2})\int|x|^{\alpha}|u|^{2^*_{\alpha}}-\mu\int u^2\log t^2\Big)\\
		&\triangleq g_1(t)t,
		\end{align*}
		and
		\begin{align*}
		g_1'(t)&=t^{2^*_{\alpha}-3}\Big(-(2^*_{\alpha}-2)\int|x|^{\alpha}|u|^{2^*_{\alpha}}-2t^{2-2^*_{\alpha}}\mu\int u^2\Big).
		\end{align*}
		Since  $u>0$ and $\mu<0$, we have that $g_1'(t)=0$ has a unique root $t_3$, and $g_1'(t)>0$ in $t\in(0,t_3)$, and $g_1'(t)<0$ in $t\in(t_3,+\infty)$. It therefore follows from $\lim\limits_{t\rightarrow0^+}g_1(t)=\lim\limits_{t\rightarrow+\infty}g_1(t)=-\infty$ that $g_1(t)=0$ has at most two roots. Thus, we have that $g'(t)=0$ has at most two roots, which implies that  $g(t)$ has at most two extreme points.
	\end{proof}
	
	\begin{lemma}\label{lemma24}{\rm(see \cite{Deng})}
		Let  $\left\{u_{n}\right\}$  be a bounded sequence in  $H_{0,r}^{1}(\Omega)$  such that  $u_{n} \rightarrow u$  a.e. in  $\Omega$, as  $n \rightarrow \infty$, then
		\begin{equation}
		\lim\limits_{n \rightarrow \infty} \int_{\Omega} u_{n}^{2} \log u_{n}^{2} \mathrm{~d} x=\int_{\Omega} u^{2} \log u^{2} \mathrm{~d} x,
		\end{equation}
		and
		\begin{equation}\label{eq216}
		\lim\limits_{n \rightarrow \infty} \int_{\Omega}\left(u_{n}\right)_{+}^{2} \log \left(u_{n}\right)_{+}^{2} \mathrm{~d} x=\int_{\Omega} u_{+}^{2} \log u_{+}^{2} \mathrm{~d} x .
		\end{equation}
	\end{lemma}
	\begin{lemma}\label{lemma25}
		Assume that  $N\geq3$ and $\lambda\in\R$.
		If $c<\min\{0, \tilde{C}_{\kappa}\}+\frac{\alpha+2}{2(N+\alpha)}S_\alpha^\frac{N+\alpha}{\alpha+2}$, then $I(u)$ satisfies the $(PS)_c$ condition.
	\end{lemma}
	\begin{proof}
		Let $\{u_n\}$ be a $(PS)_c$ sequence of $I$.  By Lemma \ref{lemma23},  we know that $\{u_n\}$ is bounded in $H_{0,r}^1(\Omega)$. Therefore, there exists a  $u\in H_{0,r}^1(\Omega)$ such that, up to a subsequence,
		\begin{align}
		\begin{array}{lll}
		u_n\rightharpoonup u~~\hbox{~weakly~in}~H_{0,r}^1(\Omega),\\
		u_n\rightarrow u~~\hbox{strongly~in}~L^q(\Omega;|x|^\alpha),~1\leq q<2^*_{\alpha}, \\
		u_n\rightarrow u~~\hbox{strongly~in}~L^p(\Omega),~1\leq p<2^*, \\
		u_n\rightarrow u~~\hbox{a.e.}~\hbox{in}~\Omega.
		\end{array}
		\end{align}
		Since $\langle I'(u_n),\varphi\rangle\rightarrow0$ as $n\rightarrow\infty$ for any $\varphi\in C_{0,r}^\infty(\Omega)$, $u$ is a weak solution to
		$$-\Delta u=|x|^\alpha|u_+|^{2^*_{\alpha}-2}u_++\lambda u_++\mu u_+\log u_+^2,$$
		which implies that
		$$u\in \kappa,$$
		$$\int |\nabla u|^2=\int|x|^\alpha|u_+|^{2^*_{\alpha}}+\lambda\int u_+^2+\mu\int u_+^2\log u_+^2,$$
		and
		\begin{equation}\label{eq217}
		\begin{split}
		I(u)&=\frac{1}{2}\int |\nabla u|^2-\frac{1}{2^*_{\alpha}}\int |x|^\alpha|u_+|^{2^*_{\alpha}}-\frac{\lambda}{2}\int u_+^2-\frac{\mu}{2}\int u_+^2(\log u_+^2-1)\\
		&=\frac{1}{2}\Big(\int|x|^\alpha|u_+|^{2^*_{\alpha}}+\lambda\int u_+^2+\mu\int u_+^2\log u_+^2\Big)-\frac{1}{2^*_{\alpha}}\int |x|^\alpha|u_+|^{2^*_{\alpha}}\\
		&-\frac{\lambda}{2}\int u_+^2-\frac{\mu}{2}\int u_+^2(\log u_+^2-1)\\
		&=\frac{\alpha+2}{2(N+\alpha)}\int |x|^\alpha|u_+|^{2^*_{\alpha}}+\frac{\mu}{2}\int u_+^2\\
		&\geq\left\{\begin{array}{ll}
		0,~\hbox{for}~\mu\geq 0,\\
		\tilde{C}_\kappa,~\hbox{for}~\mu<0.\\
		\end{array}
		\right.
		\end{split}
		\end{equation}
		Therefore, $$I(u)\geq \min\{0, \tilde{C}_\kappa\}.$$

		By the definition of $(PS)_c$, we have
		$$\int |\nabla u_n|^2-\int |x|^\alpha|(u_n)_+|^{2^*_{\alpha}}-\lambda\int (u_n)_+^2-\mu\int (u_n)_+^2\log (u_n)_+^2=o_n(1),$$
		and
		$$\frac{1}{2}\int |\nabla u_n|^2-\frac{1}{2^*_{\alpha}}\int |x|^\alpha|(u_n)_+|^{2^*_{\alpha}}-\frac{\lambda}{2}\int (u_n)_+^2-\frac{\mu}{2}\int (u_n)_+^2(\log (u_n)_+^2-1)=c+o_n(1).$$
		Setting  $v_n=u_n-u$, then
		$$\int |\nabla v_n|^2-\int |x|^\alpha|(v_n)_+|^{2^*_{\alpha}}=o_n(1),$$
		and
		\begin{equation}\label{eq222}
		I(u)+\frac{1}{2}\int |\nabla v_n|^2-\frac{1}{2^*_{\alpha}}\int |x|^\alpha|(v_n)_+|^{2^*_{\alpha}}=c+o_n(1).
		\end{equation}
		Let
		$$\int |\nabla v_n|^2\rightarrow k,~\hbox{as}~n\rightarrow\infty.$$
		So $$\int |x|^\alpha|(v_n)_+|^{2^*_{\alpha}}\rightarrow k,~\hbox{as}~n\rightarrow\infty.$$
		By the definition of $S_\alpha$, we have
		$$\int |\nabla u|^2\geq S_\alpha \Big(\int|x|^\alpha|u|^{2^*_{\alpha}}\Big)^\frac{2}{2^*_{\alpha}},~~\forall u\in H_{0,r}^1(\Omega),$$
		which implies that  $$k+o_n(1)=\int |\nabla v_n|^2\geq S_\alpha\Big(\int|x|^\alpha|(v_n)_+|^{2^*_{\alpha}}\Big)^\frac{2}{2^*_{\alpha}}=S_\alpha k^{\frac{N-2}{N+\alpha}} +o_n(1).$$
		If $k>0$, then $ k\geq S_\alpha^{{\frac{N+\alpha}{\alpha+2}}}$. By $\eqref{eq222}$, we have
		\begin{align*}
		\min\{0, \tilde{C}_\kappa\}\leq I(u)=\lim\limits_{n\rightarrow\infty}\Big(c-\frac{\alpha+2}{2(N+\alpha)}\int |\nabla v_n|^2\Big)
		=c-\frac{\alpha+2}{2(N+\alpha)}k\leq c-\frac{\alpha+2}{2(N+\alpha)}S_\alpha^\frac{N+\alpha}{\alpha+2},
		\end{align*}
		which is impossible. So $k=0$, which implies  that
		$$u_n\rightarrow u~~\hbox{in}~H_{0,r}^1(\Omega).$$
	\end{proof}

	\section{Estimations on $c_{M}$}
	
	In this section, we assume $\mu>0$. We let $\varphi\in C_0^\infty(\Omega)$  be a radial function satisfying that $\varphi(x)=1$ for $0\leq|x|\leq\rho$, $0\leq \varphi(x)\leq1$ for $\rho\leq|x|\leq2\rho$, $\varphi(x)=0$ for $x\in \Omega\backslash B_{2\rho}(0)$, where $0<\rho\leq 1$. And define
	\begin{equation}\label{eq31}
	U_{\varepsilon,\alpha}(x)=\varphi(x)u_{\varepsilon,\alpha}(x).
	\end{equation}
	\begin{lemma}\label{lemma31}
		If $N\geq3$, then we have, as $\varepsilon\rightarrow0^+$,
		\begin{equation}\label{eq32}
		\int_\Omega|\nabla U_{\varepsilon,\alpha}|^2=S_\alpha^\frac{N+\alpha}{2+\alpha}+O(\varepsilon^{N-2}),
		\end{equation}
		\begin{equation}\label{eq33}
		\int_\Omega|x|^\alpha|U_{\varepsilon,\alpha}|^{2^*_{\alpha}}=S_\alpha^\frac{N+\alpha}{2+\alpha}+O(\varepsilon^{N+\alpha}),
		\end{equation}
		\begin{gather}\label{eq34}
		\int_{\Omega}|x|^{\alpha}\left|U_{\varepsilon,\alpha}\right|^q=
		\begin{cases}
		C_{\alpha,N}^q\omega_N\varepsilon^{\frac{N+\alpha}{2}}\log \frac{1}{\varepsilon}+O(\varepsilon^{\frac{N+\alpha}{2}}),~&\mathrm{if}~ q=\frac{N+\alpha}{N-2},\\
		C_1\varepsilon^{N+\alpha-\frac{q(N-2)}{2}}+o(\varepsilon^{N+\alpha-\frac{q(N-2)}{2}}),~&\mathrm{if}~ q\in (\frac{N+\alpha}{N-2},2_\alpha^*),\\
		C_2\varepsilon^{\frac{q(N-2)}{2}}+o(\varepsilon^{\frac{q(N-2)}{2}}),~&\mathrm{if}~ 1\leq q<\frac{N+\alpha}{N-2},
		\end{cases}
		\end{gather}
		and
		\begin{gather}\label{eq53}
		\int_{\Omega}\left|U_{\varepsilon,\alpha}\right|^p=
		\begin{cases}
		C_{\alpha,N}^q\omega_N\varepsilon^{\frac{N}{2}}\log\frac{1}{\varepsilon}+O(\varepsilon^{\frac{N}{2}}),~&\mathrm{if}~ p=\frac{N}{N-2},\\
		C_3\varepsilon^{N-\frac{p(N-2)}{2}}+o(\varepsilon^{N-\frac{p(N-2)}{2}}),~&\mathrm{if}~ p\in (\frac{N}{N-2}, 2^*),\\
		C_4\varepsilon^{\frac{p(N-2)}{2}}+o(\varepsilon^{\frac{p(N-2)}{2}}),~&\mathrm{if}~ 1\leq p<\frac{N}{N-2},\\
		\end{cases}
		\end{gather}
		where $C_i, i=1,2, 3, 4$ are positive constants.
		Particularly,
		\begin{gather}\label{eq35}
		\int_{\Omega}\left|U_{\varepsilon,\alpha}\right|^2=
		\begin{cases}
		C\varepsilon+o(\varepsilon),~&\mathrm{if}~N=3,\\
		C^2_{\alpha,4}\omega_4\varepsilon^2\log\frac{1}{\varepsilon}+O(\varepsilon^2),~&\mathrm{if}~ N=4,\\
		C\varepsilon^2+o(\varepsilon^2),~&\mathrm{if}~ N\ge5,
		\end{cases}
		\end{gather}
		$$\int_{\Omega}|x|^{\alpha}\left|U_{\varepsilon,\alpha}\right|^q\leq C\varepsilon^{\min\{\frac{q(N-2)}{2},N+\alpha-\frac{q(N-2)}{2}\}}\log\frac{1}{\varepsilon}, ~~q\in[1,2^*_\alpha),$$
		and,
		$$\int_{\Omega}\left|U_{\varepsilon,\alpha}\right|^p\leq C \varepsilon^{\min\{{\frac{p(N-2)}{2}},{N-\frac{p(N-2)}{2}}\}}\log\frac{1}{\varepsilon}, ~p\in[1,2^*),$$
		where all the constants $C$ are positive constants.
	\end{lemma}
	\begin{proof}
		\eqref{eq32} and \eqref{eq33} have  been shown  in   \cite{Wangc}. So  it remains to prove \eqref{eq34} and \eqref{eq53}. Next, we only show \eqref{eq34}, and one can prove  \eqref{eq53}   by  a similar method.

		When $r\geq1$,
		by Lagrange's Mean Value Theorem,
		\begin{equation}\label{in6251}
		r^{N+\alpha-1}(1+r^{2+\alpha})^{-\frac{p(N-2)}{2+\alpha}}=r^{N+\alpha-1-p(N-2)}\Big(1-\frac{p(N-2)}{2+\alpha}(1+\xi_1)^{-\frac{p(N-2)}{2+\alpha}-1}\frac{1}{r^{2+\alpha}}\Big),
		\end{equation}
		where $\xi_1\in(0,1)$. Thus, when $q\geq \frac{N+\alpha}{N-2}$,
		\begin{align}\label{721}
		\int_{\Omega}|x|^{\alpha}\left|U_{\varepsilon,\alpha}\right|^q
		&=C_{\alpha,N}^q\int_\Omega|x|^{\alpha}\varphi^q(x)\frac{\varepsilon^{\frac{q(N-2)}{2}}}
		{(\varepsilon^{2+\alpha}+|x|^{2+\alpha})^{\frac{q(N-2)}{2+\alpha}}}\notag\\
		&=C_{\alpha,N}^q\int_{B_{\rho}(0)}|x|^{\alpha}\frac{\varepsilon^{\frac{q(N-2)}{2}}}{(\varepsilon^{2+\alpha}+|x|^{2+\alpha})^{\frac{q(N-2)}{2+\alpha}}}
		+C\varepsilon^{\frac{q(N-2)}{2}}+O(\varepsilon^{\frac{q(N-2)}{2}+2+\alpha})\notag\\
		&=C_{\alpha,N}^q\varepsilon^{N+\alpha-\frac{(N-2)q}{2}}\int_{B_{\frac{\rho}{\varepsilon}}(0)}\frac{|x|^{\alpha}}{(1+|x|^{2+\alpha})^{\frac{(N-2)q}{2+\alpha}}}
		+C\varepsilon^{\frac{q(N-2)}{2}}+O(\varepsilon^{\frac{q(N-2)}{2}+2+\alpha})\notag\\
		&=C_{\alpha,N}^q\varepsilon^{N+\alpha-\frac{(N-2)q}{2}}\omega_N\int_0^1\frac{r^{N+\alpha-1}}{(1+r^{2+\alpha})^{\frac{(N-2)q}{2+\alpha}}}~~\mathrm{d}r\notag\\
		&+C_{\alpha,N}^q\varepsilon^{N+\alpha-\frac{(N-2)q}{2}}\omega_N\int_1^{\frac{\rho}{\varepsilon}}
		\frac{r^{N+\alpha-1}}{(1+r^{2+\alpha})^{\frac{(N-2)q}{2+\alpha}}}~~\mathrm{d}r+C\varepsilon^{\frac{q(N-2)}{2}}+O(\varepsilon^{\frac{q(N-2)}{2}+2+\alpha})\notag\\
		&=C_{\alpha,N}^q\varepsilon^{N+\alpha-\frac{(N-2)q}{2}}\omega_N\int_1^{\frac{\rho}{\varepsilon}}
		\frac{r^{N+\alpha-1}}{(1+r^{2+\alpha})^{\frac{(N-2)q}{2+\alpha}}}~~\mathrm{d}r+C\varepsilon^{N+\alpha-\frac{(N-2)q}{2}}\notag\\
		&+C\varepsilon^{\frac{q(N-2)}{2}}+O(\varepsilon^{\frac{q(N-2)}{2}+2+\alpha})\notag\\
		&=C_{\alpha,N}^q\varepsilon^{N+\alpha-\frac{(N-2)q}{2}}\omega_N\int_1^{\frac{\rho}{\varepsilon}}
		r^{N+\alpha-1-q(N-2)}\Big(1-\frac{q(N-2)}{2+\alpha}(1+\xi_1)^{-\frac{q(N-2)}{2+\alpha}-1}\frac{1}{r^{2+\alpha}}\Big)~~\mathrm{d}r\notag\\
		&+C\varepsilon^{N+\alpha-\frac{(N-2)q}{2}}+C\varepsilon^{\frac{q(N-2)}{2}}+O(\varepsilon^{\frac{q(N-2)}{2}+2+\alpha})\notag\\
		&\triangleq\hbox{I}_1+C\varepsilon^{N+\alpha-\frac{(N-2)q}{2}}+C\varepsilon^{\frac{q(N-2)}{2}}+O(\varepsilon^{\frac{q(N-2)}{2}+2+\alpha}),\notag\\
		\end{align}
		where we have used that
		\begin{align}\label{734}
		\int_{B_{2\rho}(0)\setminus B_\rho(0)}|x|^{\alpha}\left|U_{\varepsilon,\alpha}\right|^q
		&=C_{\alpha,N}^q\int_{\rho\leq |x|\leq 2\rho}|x|^{\alpha}\varphi^q(x)\frac{\varepsilon^{\frac{q(N-2)}{2}}}{(\varepsilon^{2+\alpha}+|x|^{2+\alpha})^{\frac{q(N-2)}{2+\alpha}}}\notag\\
		&=C_{\alpha,N}^q\varepsilon^{\frac{q(N-2)}{2}}\int_{\rho\leq |x|\leq 2\rho}|x|^{\alpha}\varphi^q(x)\frac{1}{(\varepsilon^{2+\alpha}+|x|^{2+\alpha})^{\frac{q(N-2)}{2+\alpha}}}\notag\\
		&=C_{\alpha,N}^q\varepsilon^{\frac{q(N-2)}{2}}\int_{\rho\leq |x|\leq 2\rho}|x|^{\alpha}\varphi^q(x)\Big(|x|^{-{q(N-2)}}+O(\varepsilon^{2+\alpha})\Big)\notag\\
		&=C\varepsilon^{\frac{q(N-2)}{2}}+O(\varepsilon^{\frac{q(N-2)}{2}+2+\alpha}).\notag\\
		\end{align}

		If $q=\frac{N+\alpha}{N-2}$, then  we have
		\begin{align*}
		\hbox{I}_1&=C_{\alpha,N}^q\varepsilon^{\frac{N+\alpha}{2}}\omega_N\int_1^{\frac{\rho}{\varepsilon}}r^{-1}
		\Big(1-\frac{N+\alpha}{2+\alpha}(1+\xi_1)^{-\frac{N+\alpha}{2+\alpha}-1}\frac{1}{r^{2+\alpha}}\Big)~~\mathrm{d}r\\
		&=C_{\alpha,N}^q\omega_N\varepsilon^{\frac{N+\alpha}{2}}\log \frac{1}{\varepsilon}+O(\varepsilon^{\frac{N+\alpha}{2}}),
		\end{align*}
		which, together with \eqref{721}, implies that \eqref{eq34} is true for the case of $q=\frac{N+\alpha}{N-2}$.
		
		If $q\neq \frac{N+\alpha}{N-2}$, then  we have
		\begin{align}\label{731}
		\hbox{I}_1&=C_{\alpha,N}^q\varepsilon^{N+\alpha-\frac{(N-2)q}{2}}\omega_N\int_1^{\frac{\rho}{\varepsilon}}
		r^{N+\alpha-1-q(N-2)}\Big(1-\frac{q(N-2)}{2+\alpha}(1+\xi_1)^{-\frac{q(N-2)}{2+\alpha}-1}\frac{1}{r^{2+\alpha}}\Big)~~\mathrm{d}r\notag\\
		&=C_{\alpha,N}^q\varepsilon^{N+\alpha-\frac{(N-2)q}{2}}\omega_N
		\Big(\frac{r^{N+\alpha-(N-2)q}}{N+\alpha-(N-2)q}
		-C\frac{r^{N+\alpha-(N-2)q-(2+\alpha)}}{N+\alpha-(N-2)q-(2+\alpha)}\Big)\Big|_1^{\frac{\rho}{\varepsilon}}\notag\\
		&=C_{\alpha,N}^q\varepsilon^{N+\alpha-\frac{(N-2)q}{2}}\omega_N
		\Big(\frac{(\frac{\rho}{\varepsilon})^{N+\alpha-(N-2)q}}{N+\alpha-(N-2)q}-\frac{1}{N+\alpha-(N-2)q}\notag\\
		&
		\quad\quad\quad\quad\quad-C\frac{(\frac{\rho}{\varepsilon})^{N+\alpha-(N-2)q-(2+\alpha)}}{N+\alpha-(N-2)q-(2+\alpha)}+C\frac{1}{N+\alpha-(N-2)q-(2+\alpha)}\Big)\notag\\
		&=C_{\alpha,N}^q\varepsilon^{N+\alpha-\frac{(N-2)q}{2}}\omega_N
		\Big(-\frac{1}{N+\alpha-(N-2)q}+C\frac{1}{N+\alpha-(N-2)q-(2+\alpha)}\Big)\notag\\
		&+C_{\alpha,N}^q\omega_N\frac{\rho^{N+\alpha-(N-2)q}}{N+\alpha-(N-2)q}\varepsilon^{\frac{(N-2)q}{2}}
		+o(\varepsilon^{\frac{(N-2)q}{2}})\notag\\
		&= \begin{cases}
		d_1\varepsilon^{N+\alpha-\frac{(N-2)q}{2}}+o(\varepsilon^{N+\alpha-\frac{(N-2)q}{2}}),~&\mathrm{if}~q>\frac{N+\alpha}{N-2},\notag\\[2mm]
		d_2\varepsilon^{\frac{(N-2)q}{2}}+o(\varepsilon^{\frac{(N-2)q}{2}}),~&\mathrm{if}~ q< \frac{N+\alpha}{N-2},\notag\\[2mm]
		\end{cases}\notag\\
		\end{align}
		which, together with $\hbox{I}_1\geq 0$, implies that $d_1\geq 0$ for $q>\frac{N+\alpha}{N-2}$ and $d_2\geq 0$ for $q<\frac{N+\alpha}{N-2}$.  Following from \eqref{721} and \eqref{731},  we can see that  when $q\neq \frac{N+\alpha}{N-2}$ and $\varepsilon>0$ is small enough, \eqref{eq34} is true.

		We complete the proof.

	\end{proof}

	\begin{lemma}\label{lemma32}
		If  $N \geq 5$, then we have, as $ \varepsilon \rightarrow 0^{+}$,
		\begin{equation}
		\int_{\Omega} U_{\varepsilon,\alpha}^{2} \log U_{\varepsilon,\alpha}^{2}= C_{0} \varepsilon^{2} \log \frac{1}{\varepsilon}+O\left(\varepsilon^{2}\right),
		\end{equation}
		where  $C_{0}$  is a positive constant.
	\end{lemma}
	\begin{proof}

		It follows from \eqref{eq31} that
		
		$$\begin{array}{ll}
		\begin{split}
		\int_{\Omega} U_{\varepsilon,\alpha}^{2} \log U_{\varepsilon,\alpha}^{2}&=\int_{\Omega} \varphi^{2} u_{\varepsilon,\alpha}^{2} \log \varphi^{2}+\int_{\Omega} \varphi^{2} u_{\varepsilon,\alpha}^{2} \log u_{\varepsilon,\alpha}^{2} \\
		&=\int_{\Omega} \varphi^{2} u_{\varepsilon,\alpha}^{2} \log \varphi^{2}+\int_{B_{\rho}(0)} u_{\varepsilon,\alpha}^{2} \log u_{\varepsilon,\alpha}^{2}+\int_{\Omega\backslash B_{\rho}(0)} \varphi^{2} u_{\varepsilon,\alpha}^{2} \log u_{\varepsilon,\alpha}^{2}\\
		& \triangleq \hbox{I}_2+\mathrm{I}_{3}+\mathrm{I}_4.
		\end{split}
		\end{array}$$
		Since  $\left|s^{2} \log s^{2}\right| \leq C$  for  $0 \leq s \leq 1$, we have
		\begin{equation}
		\begin{split}
		|\mathrm{I}_2| &\leq C \int_{\Omega} u_{\varepsilon,\alpha}^{2}=O\left(\varepsilon^{2}\right) . \\
		\end{split}
		\end{equation}
		Due to the fact that  $|s \log s| \leq C_{1} s^{1-\delta}+C_{2} s^{1+\delta}$  for all  $s>0$, where  $0<C_{1}<C_{2}$  and  $0<\delta<\frac{1}{3}$  such that  $(N-2)(1-\delta) \geq 2$, we have that
		\begin{equation}
		\begin{split}
		|\hbox{I}_4| \leq \int_{\Omega \backslash B_{\rho}(0)}\left|u_{\varepsilon,\alpha}^{2} \log u_{\varepsilon,\alpha}^{2}\right|  \leq C \int_{\Omega \backslash B_{\rho}(0)}(u_{\varepsilon,\alpha}^{2(1-\delta)}+u_{\varepsilon,\alpha}^{2(1+\delta)})
		=O(\varepsilon^{2}),
		\end{split}
		\end{equation}
		where we have used that,
		\begin{equation}
		u_{\varepsilon,\alpha}^{2}=\frac{C^2_{\alpha,N}\varepsilon^{N-2}}{\left(\varepsilon^{2+\alpha}+|x|^{2+\alpha}\right)^\frac{2(N-2)}{2+\alpha}}
		\leq \frac{C\varepsilon^{N-2}}{|x|^{2(N-2)}}
		\leq \frac{C\varepsilon^{N-2}}{\rho^{2(N-2)}}
		=C\varepsilon^{N-2},~~|x|\geq \rho.
		\end{equation}
		Letting  $x=\varepsilon y$, we have that
		\begin{align*}
		u_{\varepsilon,\alpha}^{2}&=\frac{C^2_{\alpha,N}\varepsilon^{N-2}}{\left(\varepsilon^{2+\alpha}+|x|^{2+\alpha}\right)^\frac{2(N-2)}{2+\alpha}}
		=\frac{C^2_{\alpha,N}\varepsilon^{N-2}}{\left(\varepsilon^{2+\alpha}+\varepsilon^{2+\alpha}|y|^{2+\alpha}\right)^\frac{2(N-2)}{2+\alpha}}
		=\frac{C\varepsilon^{N-2}}{\varepsilon^{2(N-2)}\left(1+|y|^{2+\alpha}\right)^\frac{2(N-2)}{2+\alpha}},
		\end{align*}
		and
		\begin{align*}
		\mathrm{I}_3 & =\int _{B_{ \rho}(0)} u_{\varepsilon,\alpha}^{2} \log u_{\varepsilon,\alpha}^{2} \\
		& =\int_{B_{\frac{\rho}{\varepsilon}(0)}} \frac{C\varepsilon^{N-2}}{\varepsilon^{2(N-2)}\left(1+|y|^{2+\alpha}\right)^\frac{2(N-2)}{2+\alpha}} \log \frac{C\varepsilon^{N-2}}{\varepsilon^{2(N-2)}\left(1+|y|^{2+\alpha}\right)^\frac{2(N-2)}{2+\alpha}} \cdot \varepsilon^{N}\mathrm{d} y \\
		&=C\varepsilon^2\int _{B_{\frac{\rho}{\varepsilon}}(0)} \frac{1}{\left(1+|y|^{2+\alpha}\right)^\frac{2(N-2)}{2+\alpha}}\log \frac{C\varepsilon^{-(N-2)}}{\left(1+|y|^{2+\alpha}\right)^\frac{2(N-2)}{2+\alpha}}\mathrm{~d} y\\
		& =C \varepsilon^{2} \log \Big(\frac{1}{\varepsilon}\Big) \int_{B_{\frac{\rho}{\varepsilon}}(0)} \frac{1}{\left(1+|y|^{\alpha+2}\right)^{\frac{2(N-2)}{N+\alpha}}}\mathrm{~d} y\\
		&+C \varepsilon^{2} \int_{B_{\frac{\rho}{\varepsilon}}(0)} \frac{1}{\left(1+|y|^{\alpha+2}\right)^{\frac{2(N-2)}{N+\alpha}}} \log \frac{C}{\left(1+|y|^{\alpha+2}\right)^{\frac{2(N-2)}{N+\alpha}}}\mathrm{~d} y \\
		& =C \varepsilon^{2} \log \Big(\frac{1}{\varepsilon}\Big)+O\left(\varepsilon^{2}\right).
		\end{align*}
		Thus,
		\begin{equation*}
		\int_{\Omega} U_{\varepsilon,\alpha}^{2} \log U_{\varepsilon,\alpha}^{2}= C_{0} \varepsilon^{2} \log \frac{1}{\varepsilon}+O\left(\varepsilon^{2}\right).
		\end{equation*}
	\end{proof}

	\begin{lemma}\label{lemma33}
		If $N\geq5,~\lambda\in\R$ and $\mu>0$, then $c_M<\min\{0, \tilde{C}_\kappa\}+\frac{\alpha+2}{2(N+\alpha)}S_\alpha^\frac{N+\alpha}{\alpha+2}$.
	\end{lemma}
	\begin{proof}When $\mu>0$, $\tilde{C}_{\kappa}\geq 0.$ That is, $\min\{0, \tilde{C}_\kappa\}=0$. Therefore, we only need to show that
		$c_M<\frac{\alpha+2}{2(N+\alpha)}S_\alpha^\frac{N+\alpha}{\alpha+2}$.

		Let $g(t)\triangleq I(tU_{\varepsilon,\alpha})$.~By Lemma \ref{lemma22}, $g(0)=0$ and $\lim\limits_{t\rightarrow+\infty}g(t)=-\infty$, we can find $t_\varepsilon\in(0,+\infty)$ such that
		$$\sup\limits_{t\geq0}I(tU_{\varepsilon,\alpha})=\sup\limits_{t\geq0}g(t)=g(t_\varepsilon)=I(t_\varepsilon U_{\varepsilon,\alpha}).$$
		So
		$\langle I'(t_\varepsilon U_{\varepsilon,\alpha}),t_\varepsilon U_{\varepsilon,\alpha}\rangle=0.$ That is,
		\begin{equation}\label{eq313}
		t_\varepsilon ^2\int |\nabla U_{\varepsilon,\alpha}|^2-t_\varepsilon ^{2^*_{\alpha}}\int |x|^\alpha|U_{\varepsilon,\alpha}|^{2^*_{\alpha}}-\lambda t_\varepsilon ^2\int U_{\varepsilon,\alpha}^2-\mu t_\varepsilon ^2\int U_{\varepsilon,\alpha}^2\log t_\varepsilon^2U_{\varepsilon,\alpha}^2=0,
		\end{equation}
		which implies that, as $\varepsilon\rightarrow0^+$,
		\begin{align*}
		2S_\alpha^\frac{N+\alpha}{2+\alpha}&\geq \int |\nabla U_{\varepsilon,\alpha}|^2-\lambda\int U_{\varepsilon,\alpha}^2-\mu\int U_{\varepsilon,\alpha}^2\log U_{\varepsilon,\alpha}^2\\
		&=t_\varepsilon^{2^*_{\alpha}-2}\int |x|^\alpha|U_{\varepsilon,\alpha}|^{2^*_{\alpha}}+\mu\int U_{\varepsilon,\alpha}^2\log t_\varepsilon^2\\
		&\geq t_\varepsilon^{2^*_{\alpha}-2}\left(\frac{1}{2}S_\alpha^\frac{N+\alpha}{2+\alpha}\right)-C|\log t_\varepsilon^2|.
		\end{align*}
		So there exists $c_1>0$ such that $t_\varepsilon<c_1$.
		
		On the other hand, as $\varepsilon\rightarrow0^+$,
		\begin{align*}
		\frac{1}{2}S_\alpha^\frac{N+\alpha}{2+\alpha}&\leq \int |\nabla U_{\varepsilon,\alpha}|^2-\lambda\int U_{\varepsilon,\alpha}^2-\mu\int U_{\varepsilon,\alpha}^2\log U_{\varepsilon,\alpha}^2\\
		&=t_\varepsilon^{2^*_{\alpha}-2}\int |x|^\alpha|U_{\varepsilon,\alpha}|^{2^*_{\alpha}}+\mu\int U_{\varepsilon,\alpha}^2\log t_\varepsilon^2\\
		&\leq 2S_\alpha^\frac{N+\alpha}{2+\alpha}t_\varepsilon^{2^*_{\alpha}-2}+Ct_\varepsilon^{2^*_{\alpha}-2}.
		\end{align*}
		Thus there exists $c_2>0$ such that $t_\varepsilon>c_2$.
		
		Therefore, combining with the definition of $c_M$, we have that, as $\varepsilon\rightarrow0^+$,
		\begin{align*}
		c_M&\leq \sup\limits_{t\geq0}I(tU_{\varepsilon,\alpha})\\
		&= \frac{t_\varepsilon^2}{2}\int|\nabla U_{\varepsilon,\alpha}|^2-\frac{t_\varepsilon^{2^*_{\alpha}}}{2^*_{\alpha}}\int|x|^\alpha|U_{\varepsilon,\alpha}|^{2^*_{\alpha}}-\frac{\lambda}{2}t_\varepsilon^2\int U_{\varepsilon,\alpha}^2-\frac{\mu}{2}t_\varepsilon^2\int U_{\varepsilon,\alpha}^2(\log (t_\varepsilon^2U_{\varepsilon,\alpha}^2)-1)\\
		&\leq \frac{t_\varepsilon^2}{2}S_\alpha^\frac{N+\alpha}{2+\alpha}-\frac{t_\varepsilon^{2^*_{\alpha}}}{2^*_{\alpha}}S_\alpha^\frac{N+\alpha}{2+\alpha}+O(\varepsilon^2)+\frac{\mu}{2}t_\varepsilon^2\left(1-\log t_\varepsilon^2\right)\int U_{\varepsilon,\alpha}^2-\frac{\mu}{2}t_\varepsilon^2\int U_{\varepsilon,\alpha}^2\log U_{\varepsilon,\alpha}^2\\
		&\leq \frac{\alpha+2}{2(N+\alpha)}S_\alpha^\frac{N+\alpha}{\alpha+2}-C
		\mu\varepsilon^{2} \log \frac{1}{\varepsilon}+O(\varepsilon^2)\\
		&<\frac{\alpha+2}{2(N+\alpha)}S_\alpha^\frac{N+\alpha}{\alpha+2}.
		\end{align*}
	\end{proof}

	\begin{lemma}\label{lm3.4}
		If $N=4$, then we have, as $\varepsilon\rightarrow0^+$,
		\begin{equation*}
		\int_\Omega U_{\varepsilon,\alpha}^2\log U_{\varepsilon,\alpha}^2\geq C^2_{\alpha,4}\omega_4\log\Big(\frac{C^2_{\alpha,4}\rho^2e^{-\frac{8}{(2+\alpha)^2}}}{(2\rho)^4}\Big)\varepsilon^2\log\frac{1}{\varepsilon}+O(\varepsilon^2),
		\end{equation*}
		where $\omega_4$ denotes the area of the unit sphere surface in $\R^4$.
	\end{lemma}
	\begin{proof}
		
		Following from the definition of $u_{\varepsilon,\alpha}$, we obtain that
		$u_{\varepsilon,\alpha}=\frac{C_{\alpha,4}\varepsilon}{(\varepsilon^{2+\alpha}+|x|^{2+\alpha})^\frac{2}{2+\alpha}},$
		and
		\begin{align}\label{ineq1}
		\int_\Omega U_{\varepsilon,\alpha}^2\log U_{\varepsilon,\alpha}^2
		&=C^2_{\alpha,4}\int_{B_2\rho(0)} \frac{\varepsilon^2}{(\varepsilon^{2+\alpha}+|x|^{2+\alpha})^\frac{4}{2+\alpha}}\varphi^2\log \frac{\varepsilon^2}{(\varepsilon^{2+\alpha}+|x|^{2+\alpha})^\frac{4}{2+\alpha}}\notag\\
		&+C^2_{\alpha,4} \log C^2_{\alpha,4}\int_{B_\rho(0)} \frac{\varepsilon^2}{(\varepsilon^{2+\alpha}+|x|^{2+\alpha})^\frac{4}{2+\alpha}}\notag\\
		&+C^2_{\alpha,4}\int_{\Omega\backslash B_\rho(0)}\frac{\varepsilon^2}{(\varepsilon^{2+\alpha}+|x|^{2+\alpha})^\frac{4}{2+\alpha}}\varphi^2\log \varphi^2+O(\varepsilon^2) \notag\\
		&=\varepsilon^2C^2_{\alpha,4}\int_{B_\frac{2\rho}{\varepsilon}(0)}\frac{1}{(1+|x|^{2+\alpha})^\frac{4}{2+\alpha}}\varphi^2(\varepsilon x)\log \frac{1}{\varepsilon^2(1+|x|^{2+\alpha})^\frac{4}{2+\alpha}}\notag\\
		&+\varepsilon^2C^2_{\alpha,4} \log C^2_{\alpha,4}\int_{B_\frac{\rho}{\varepsilon}(0)}\frac{1}{(1+|x|^{2+\alpha})^\frac{4}{2+\alpha}}+O(\varepsilon^2)\notag\\
		&=2\varepsilon^2C^2_{\alpha,4}\log\frac{1}{\varepsilon}\int_{B_\frac{2\rho}{\varepsilon}(0)}\frac{1}{(1+|x|^{2+\alpha})^\frac{4}{2+\alpha}}\varphi^2(\varepsilon x)\notag\\
		&+\varepsilon^2C^2_{\alpha,4}\int_{B_\frac{2\rho}{\varepsilon}(0)}\frac{1}{(1+|x|^{2+\alpha})^\frac{4}{2+\alpha}}\varphi^2(\varepsilon x)\log\frac{1}{(1+|x|^{2+\alpha})^\frac{4}{2+\alpha}}\notag\\
		&+\varepsilon^2C^2_{\alpha,4} \log C^2_{\alpha,4}\int_{B_\frac{\rho}{\varepsilon}(0)}\frac{1}{(1+|x|^{2+\alpha})^\frac{4}{2+\alpha}}+O(\varepsilon^2)\notag\\
		&\triangleq\hbox{II}_{1}+\hbox{II}_{2}+\hbox{II}_3+O(\varepsilon^2).
		\end{align}
		By \eqref{in6251} and  direct computations, we have that
		\begin{align}\label{eq123}
		\hbox{II}_3
		&=C^2_{\alpha,4} \log C^2_{\alpha,4}\omega_4\varepsilon^2\int_0^{\frac{\rho}{\varepsilon}}\frac{r^3}{(1+r^{2+\alpha})^\frac{4}{2+\alpha}}\mathrm{~d}r\notag\\
		&=C^2_{\alpha,4} \log C^2_{\alpha,4}\omega_4\varepsilon^2\int_0^{1}\frac{r^3}{(1+r^{2+\alpha})^\frac{4}{2+\alpha}}\mathrm{~d}r+C^2_{\alpha,4} \log C^2_{\alpha,4}\omega_4\varepsilon^2\int_{1}^{\frac{\rho}{\varepsilon}}\frac{r^3}{(1+r^{2+\alpha})^\frac{4}{2+\alpha}}\mathrm{~d}r\notag\\
		&= C^2_{\alpha,4} \log C^2_{\alpha,4}\omega_4\varepsilon^2\int_{1}^{\frac{\rho}{\varepsilon}}r^{-1}
		\Big(1-\frac{4}{2+\alpha}(1+\xi)^{-\frac{4}{2+\alpha}-1}\frac{1}{r^{2+\alpha}}\Big)\mathrm{~d}r+O(\varepsilon^2)\notag\\
		&=C^2_{\alpha,4}\log C^2_{\alpha,4}\omega_4\varepsilon^2\log\frac{\rho}{\varepsilon}+O(\varepsilon^2)\notag\\
		&=C^2_{\alpha,4}\log C^2_{\alpha,4}\omega_4\varepsilon^2\log\frac{1}{\varepsilon}+O(\varepsilon^2),
		\end{align}
		
		and
		\begin{align}\label{ineq4}
		\hbox{II}_{1}
		&\geq2\varepsilon^2C^2_{\alpha,4}\log\frac{1}{\varepsilon}\omega_4\int_1^{\frac{\rho}{\varepsilon}}\frac{r^3}{(1+r^{2+\alpha})^\frac{4}{2+\alpha}}~~\mathrm{d}r\notag\\
		&=2\varepsilon^2C^2_{\alpha,4}\log\frac{1}{\varepsilon}\omega_4\Big(\log r+\frac{4}{(2+\alpha)^2}(1+\xi)^{-\frac{4}{2+\alpha}-1}\frac{1}{r^{2+\alpha}}\Big)
		\Big|_1^{\frac{\rho}{\varepsilon}}\notag\\
		&=2\varepsilon^2C^2_{\alpha,4}\omega_4(\log \frac{1}{\varepsilon})^2+2\varepsilon^2C^2_{\alpha,4}\omega_4\log \frac{1}{\varepsilon}\Big(\log\Big(\rho e^{-\frac{4}{(2+\alpha)^2}(1+\xi)^{-\frac{4}{2+\alpha}-1}}\Big)\Big)+O(\varepsilon^{4+\alpha}\log\frac{1}{\varepsilon})\notag\\
		&\geq 2\varepsilon^2C^2_{\alpha,4}\omega_4(\log \frac{1}{\varepsilon})^2+2\varepsilon^2C^2_{\alpha,4}\omega_4\log \frac{1}{\varepsilon}\Big(\log\Big(\rho e^{-\frac{4}{(2+\alpha)^2}}\Big)\Big)+O(\varepsilon^{4+\alpha}\log\frac{1}{\varepsilon}),
		\end{align}
		where $ \xi \in (0,1)$.

		If $r\geq r_0\geq 1$ with  $|\frac{4}{2+\alpha}\frac{1}{r_0^{2+\alpha}}|<\frac{1}{2}$, then,   by the Lagrange's Mean Value Theorem,  we have that
		\begin{equation*}
		\begin{split}
		&r^3(1+r^{2+\alpha})^{-\frac{4}{2+\alpha}}\log\Big((1+r^{2+\alpha})^{-\frac{4}{2+\alpha}}\Big)\\
		&=r^{-1}\Big(1-\frac{4}{2+\alpha}(1+\xi)^{-\frac{4}{2+\alpha}-1}\frac{1}{r^{2+\alpha}}\Big)\log\Big(r^{-4}-\frac{4}{2+\alpha}(1+\xi)^{-\frac{4}{2+\alpha}-1}\frac{1}{r^{4+(2+\alpha)}}\Big)\\
		&=r^{-1}\Big(1-\frac{4}{2+\alpha}(1+\xi)^{-\frac{4}{2+\alpha}-1}\frac{1}{r^{2+\alpha}}\Big)\Big(-4\log r+\log\Big(1-\frac{4}{2+\alpha}(1+\xi)^{-\frac{4}{2+\alpha}-1}\frac{1}{r^{2+\alpha}}\Big)\Big)\\
		&\geq r^{-1}\Big(1-\frac{4}{2+\alpha}(1+\xi)^{-\frac{4}{2+\alpha}-1}\frac{1}{r^{2+\alpha}}\Big)\Big(-4\log r-C\frac{1}{r^{2+\alpha}}\Big)\Big)\\
		&\geq -4r^{-1}\log r-C\frac{1}{r^{3+\alpha}},
		\end{split}
		\end{equation*}
		where $\xi  \in(0,1), C>\frac{4}{2+\alpha} $,  and which implies that
		\begin{align}\label{ineq5}
		\hbox{II}_{2}
		&\geq\varepsilon^2C^2_{\alpha,4}\int_{B_\frac{2\rho}{\varepsilon}(0)}\frac{1}{(1+|x|^{2+\alpha})^\frac{4}{2+\alpha}}\log\frac{1}{(1+|x|^{2+\alpha})^\frac{4}{2+\alpha}}\notag\\
		&=\varepsilon^2C^2_{\alpha,4}\omega_4\int_{r_0}^{\frac{2\rho}{\varepsilon}}\frac{r^3}
		{(1+r^{2+\alpha})^\frac{4}{2+\alpha}}\log\frac{1}{(1+r^{2+\alpha})^\frac{4}{2+\alpha}}~~\mathrm{d}r+ O(\varepsilon^2)\notag\\
		&\geq\varepsilon^2C^2_{\alpha,4}\omega_4\int_{r_0}^{\frac{2\rho}{\varepsilon}}\Big(-4r^{-1}\log r-C\frac{1}{r^{1+(2+\alpha)}}\Big)~~\mathrm{d}r +O(\varepsilon^2)\notag\\
		&=\varepsilon^2C^2_{\alpha,4}\omega_4\Big(-2(\log r)^2-C\frac{1}{r^{2+\alpha}}\Big)\Big|_{r_0}^{\frac{2\rho}{\varepsilon}}+O(\varepsilon^2)\notag\\
		&=-2C^2_{\alpha,4}\omega_4\varepsilon^2\Big(\log\frac{2\rho}{\va}\Big)^2+ O(\varepsilon^2)\notag\\
		&=-2C^2_{\alpha,4}\omega_4\varepsilon^2\Big(\log\frac{1}{\varepsilon}\Big)^2-4C^2_{\alpha,4}\omega_4\log(2\rho)\varepsilon^2\log\frac{1}{\varepsilon}+O(\varepsilon^2).
		\end{align}
		Following from \eqref{ineq1}--\eqref{ineq5},  we have that
		\begin{equation*}
		\int_\Omega U_{\varepsilon,\alpha}^2\log U_{\varepsilon,\alpha}^2\geq C^2_{\alpha,4}\omega_4\log\Big(\frac{C^2_{\alpha,4}\rho^2e^{-\frac{8}{(2+\alpha)^2}}}{(2\rho)^4}\Big)\varepsilon^2\log\frac{1}{\varepsilon}+O(\varepsilon^2).
		\end{equation*}
	\end{proof}

	\begin{lemma}\label{lemma35}
		Assume that $N=4$. If $\lambda\in\R$ and $\mu>0$, then $c_M<\min\{0, \tilde{C}_\kappa\}+\frac{\alpha+2}{2(N+\alpha)}S_\alpha^\frac{N+\alpha}{\alpha+2}$.
	\end{lemma}
	\begin{proof}
		Similar to the case of $N\geq5$, we only need to show that $c_M<\frac{\alpha+2}{2(N+\alpha)}S_\alpha^\frac{N+\alpha}{\alpha+2}$, and   we can find $0<C_1<C_2<+\infty$ and  $t_\varepsilon\in(C_1, C_2)$ such that
		$$\sup\limits_{t\geq0}I(tU_{\varepsilon,\alpha})=I(t_\varepsilon U_{\varepsilon,\alpha})$$
		and
		\begin{equation}\label{eq313}
		t_\varepsilon ^2\int |\nabla U_{\varepsilon,\alpha}|^2-t_\varepsilon ^{2^*_{\alpha}}\int |x|^\alpha|U_{\varepsilon,\alpha}|^{2^*_{\alpha}}-\lambda t_\varepsilon ^2\int U_{\varepsilon,\alpha}^2-\mu t_\varepsilon ^2\int U_{\varepsilon,\alpha}^2\log t_\varepsilon^2U_{\varepsilon,\alpha}^2=0.
		\end{equation}
		From Lemma \ref{lemma31} and \eqref{eq313}, we obtain
		\begin{equation*}
		\mu\log t_\varepsilon^2\int U_{\varepsilon,\alpha}^2=O(\varepsilon^2|\log\varepsilon|),
		\end{equation*}
		and
		\begin{align*}
		t_\varepsilon^{2^*_{\alpha}-2}&=\frac{\int |\nabla U_{\varepsilon,\alpha}|^2-\lambda \int U_{\varepsilon,\alpha}^2-\mu \int U_{\varepsilon,\alpha}^2\log U_{\varepsilon,\alpha}^2-\mu\log t_\varepsilon^2\int U_{\varepsilon,\alpha}^2}{\int |x|^\alpha|U_{\varepsilon,\alpha}|^{2^*_{\alpha}}}\\
		&=\frac{S_\alpha^\frac{N+\alpha}{\alpha+2}+O(\varepsilon^2|\log\varepsilon|)}{S_\alpha^\frac{N+\alpha}{\alpha+2}+O(\varepsilon^{N+\alpha})}\rightarrow 1\quad\hbox{as}~\varepsilon\rightarrow0^+,
		\end{align*}
		which implies that
		\begin{equation*}
		\mu\log t_\varepsilon^2\int U_{\varepsilon,\alpha}^2=o(\varepsilon^2|\log\varepsilon|).
		\end{equation*}
		According to \eqref{eq35}, we obtain
		\begin{equation*}
		\int U_{\varepsilon,\alpha}^2=C_{\alpha,4}^2\omega_4\varepsilon^2\log\frac{1}{\varepsilon}+O(\varepsilon^2).
		\end{equation*}
		Therefore, we have that, for $\varepsilon$ small enough,
		\begin{align*}
		c_M&\leq \sup\limits_{t\geq0}I(tU_{\varepsilon,\alpha})\\
		&= \frac{t_\varepsilon^2}{2}\int|\nabla U_{\varepsilon,\alpha}|^2-\frac{t_\varepsilon^{2^*_{\alpha}}}{2^*_{\alpha}}\int|x|^\alpha|U_{\varepsilon,\alpha}|^{2^*_{\alpha}}-\frac{\lambda}{2}t_\varepsilon^2\int U_{\varepsilon,\alpha}^2-\frac{\mu}{2}t_\varepsilon^2\int U_{\varepsilon,\alpha}^2(\log (t_\varepsilon^2U_{\varepsilon,\alpha}^2)-1)\\
		&\leq \frac{t_\varepsilon^2}{2}S_\alpha^\frac{N+\alpha}{\alpha+2}-\frac{t_{\varepsilon}^{2^*_{\alpha}}}{2^*_{\alpha}}S_\alpha^\frac{N+\alpha}{\alpha+2}+O(\varepsilon^2)+\frac{\mu-\lambda}{2}t_{\varepsilon}^2\int U_{\varepsilon,\alpha}^2-\frac{\mu}{2}t_{\varepsilon}^2\int U_{\varepsilon,\alpha}^2\log U_{\varepsilon,\alpha}^2+o(\varepsilon^2|\log\varepsilon|)\\
		&\leq \frac{\alpha+2}{2(N+\alpha)}S_\alpha^\frac{N+\alpha}{\alpha+2}-\frac{t_\varepsilon^2}{2}\int \Big[\mu U_{\varepsilon,\alpha}^2\log U_{\varepsilon,\alpha}^2+(\lambda-\mu)U_{\varepsilon,\alpha}^2\Big]+o(\varepsilon^2|\log\varepsilon|)\\
		&\leq \frac{\alpha+2}{2(N+\alpha)}S_\alpha^\frac{N+\alpha}{\alpha+2}+o(\varepsilon^2|\log\varepsilon|)\\
		&-\frac{t_\varepsilon^2}{2}\Big(\mu \varepsilon^2C^2_{\alpha,4}\omega_4\log\frac{1}{\varepsilon}\log\Big(\frac{C^2_{\alpha,4}\rho^2e^{-\frac{8}{(2+\alpha)^2}}}{(2\rho)^4}\Big)+(\lambda-\mu)C_{\alpha,4}^2\omega_4\varepsilon^2\log\frac{1}{\varepsilon}\Big)\\
		&\leq \frac{\alpha+2}{2(N+\alpha)}S_\alpha^\frac{N+\alpha}{\alpha+2}
		-\frac{t_\varepsilon^2}{2}C_{\alpha,4}^2\omega_4\varepsilon^2\log\frac{1}{\varepsilon}
		\log\Big(\frac{C^{2\mu}_{\alpha,4}\rho^{2\mu}e^{-\frac{8\mu}{(2+\alpha)^2}+\lambda-\mu}}{(2\rho)^{4\mu}}\Big)+o(\varepsilon^2|\log\varepsilon|)\\
		&<\frac{\alpha+2}{2(N+\alpha)}S_\alpha^\frac{N+\alpha}{\alpha+2},
		\end{align*}
		where we choose $\rho>0$ small enough such that $\frac{C^{2\mu}_{\alpha,4}\rho^{2\mu}e^{-\frac{8\mu}{(2+\alpha)^2}+\lambda-\mu}}{(2\rho)^{4\mu}}>1$.
	\end{proof}
	
	\begin{proof}[\bf Proof of Theorem \ref{Th41}:]
		Assume that  $N \geq 4$  and  $(\lambda, \mu) \in A_{0}$.~By Lemma \ref{lemma22} and the Mountain pass Theorem, there exists a sequence  $\left\{u_{n}\right\} \subset H_{0,r}^{1}(\Omega)$  such that, as  $n \rightarrow \infty$,
		
		$$I\left(u_{n}\right) \rightarrow c_{M}, \quad I^{\prime}\left(u_{n}\right) \rightarrow 0, \quad \text { in }\left(H_{0,r}^{1}(\Omega)\right)^{-1},$$
		which, combining  Lemmas \ref{lemma23},  \ref{lemma25},   \ref{lemma33} and   \ref{lemma35}, implies that  that there exists  $u \in H_{0,r}^{1}(\Omega)$  such that  $u_{n} \rightarrow u$  in  $H_{0,r}^{1}(\Omega)$. So
		$$I(u)=c_{M} \quad \text { and } \quad I^{\prime}(u)=0.$$
		Thus,
		$$0=\langle I^{\prime}(u), u_{-}\rangle=\int|\nabla u_{-}|^{2},$$
		which implies that  $u \geq 0$.
		Therefore,  $u$  is a non-negative non-trivial weak solution of \eqref{eq11}. By Moser's iteration, it is standard to prove that  $u \in L^{\infty}(\Omega)$. Then the H\"{o}lder estimate implies that  $u \in C^{0, \gamma}_{loc}(\Omega)(0<\gamma<1)$. Let  $\beta:[0,+\infty) \mapsto \mathbb{R}$  be defined by
		
		$$\beta(s):=\left\{\begin{array}{ll}
		\frac{3|\mu|}{2}\left|s \log s^{2}\right|, & s>0, \\
		0, & s=0,
		\end{array}\right.$$
		then for  $a>0$  small enough, one can see that
		$$\Delta u=-|x|^\alpha u^{2^*_{\alpha}-1}-\lambda u-\mu u \log u^{2} \leq \beta(u), \quad \text { in }\{x \in \Omega: 0<u(x)<a\}.$$
		We may also assume that  $a<\frac{1}{e}$. Then  $\beta^{\prime}(s)=\frac{3|\mu|}{2}\left(-\log s^{2}-2\right)>|\mu|(-\log a-1)>0$  for  $s \in(0, a)$. So we have that  $\beta(0)=0$  and  $\beta(s)$  is nondecreasing in  $(0,a)$. Furthermore,
		$$\int_{0}^{\frac{a}{2}}(\beta(s) s)^{-\frac{1}{2}} \mathrm{~d} s=-\sqrt{\frac{2}{3|\mu|}}(-2 \log s)^{\frac{1}{2}}\Big|_{0} ^{\frac{a}{2}}=+\infty.$$
		Hence, by \cite[Theorem 1]{vjl}, we have that  $u(x)>0$  in  $\Omega$.~In particular, for any compact  $K \subset \subset \Omega$, there exists $c=c(K)>0$  such that $u(x) \geq c, ~\forall x \in K$. Taking  $K \subset \subset K_{1} \subset \subset \Omega$  and putting   $f(x):=-|x|^\alpha u(x)^{2^*_{\alpha}-1}-\lambda u(x)-\mu u(x) \log u(x)^{2}$, then $\Delta u=f(x)$  in  $K_{1}$ and $f$  is of $ C^{0, \gamma }$  in  $K_{1}$. So by the standard Schauder estimate, we see that  $u \in C^{2, \gamma}(K)$. By the arbitrariness of  $K$, we obtain that  $u \in C^{2}(\Omega)$  and  $u>0$  in  $\Omega$.
	\end{proof}
	
	\section{Finding a  least energy solution  for the case  of $\mu<0$}
	
	\begin{proof}[\bf Proof of Theorem \ref{Th61}:] We divide the proof into three cases:
		
		\textbf{Step  1:} We show that $\tilde{C}_{\rho}$ can be attained  by some $\bar{u}\in A:=\{u\in H^1_{0,r}(\Omega):|\nabla u|_2<\rho\}$ with $\bar{u}\in C^2(\Omega)$ and $\bar{u}>0$.

		By Lemma \ref{lemma51}, we obtain $-\infty<\tilde{C}_{\rho}<0$. By Lemma \ref{lemma22}, we can get a minimizing sequence $\{u_n\}$ for $\tilde{C}_{\rho}$ with $|\nabla u_n|_2<\rho-\tau$ and $\tau>0$ small enough.  By Ekeland's
		variational principle(see \cite{m}), we known that there exists a sequence $\{\bar{u}_n\}\subset H_{0,r}^1(\Omega)$ such that $||u_n-\bar{u}_n||\rightarrow0$, $I(\bar{u}_n)\rightarrow\tilde{C}_{\rho}$ and $I'(\bar{u}_n)\rightarrow0$. Due to Lemma \ref{lemma23},
		we can see that $\{\bar{u}_n\}$ is bounded in $H_{0,r}^1(\Omega)$.
		Hence, we may assume that
		\begin{align}
		\begin{array}{lll}
		\bar{u}_n\rightharpoonup \bar{u}~~\hbox{weakly~in}~H_{0,r}^1(\Omega),\\
		\bar{u}_n\rightarrow \bar{u}~~\hbox{strongly~in}~L^q(\Omega;|x|^\alpha),~1\leq q<2^*_{\alpha}, \\
		\bar{u}_n\rightarrow \bar{u}~~\hbox{strongly~in}~L^p(\Omega),~1\leq p<2^*, \\
		\bar{u}_n\rightarrow \bar{u}~~\hbox{a.e.}~\hbox{in}~\Omega.
		\end{array}
		\end{align}
		By the weak-lower semi-continuity of the norm, we see that $|\nabla \bar{u}|_2<\rho$. By $I'(\bar{u}_n)\rightarrow0$, we have that $I'(\bar{u})=0$ and $\bar{u}\geq 0$.
		Letting $\omega_n=\bar{u}_n-\bar{u}$, we obtain
		\begin{equation*}
		\int |\nabla \omega_n|^2-\int |x|^\alpha|(\omega_n)_+|^{2^*_{\alpha}}=o_n(1),
		\end{equation*}
		and
		\begin{equation*}
		I(\bar{u}_n)=I(\bar{u})+\frac{\alpha+2}{2(N+\alpha)}\int|\nabla \omega_n|^2+o_n(1).
		\end{equation*}
		Let $\int|\nabla \omega_n|^2\rightarrow k$, we have that, as $n\rightarrow\infty$,
		\begin{align*}
		\tilde{C}_{\rho}\leq I(\bar{u})&\leq I(\bar{u})+\frac{\alpha+2}{2(N+\alpha)}k
		=\lim\limits_{n\rightarrow\infty}I(\bar{u}_n)=\tilde{C}_{\rho},
		\end{align*}
		where we have used the fact that $|\nabla \bar{u}|_2<\rho$ and which implies that $k=0$.~Hence, we obtain
		$$\bar{u}_n\rightarrow \bar{u} ~\hbox{strongly in}~ H_{0,r}^1(\Omega)$$
		and
		$$I( \bar{u})=\tilde{C}_{\rho}.$$

		Since $I( \bar{u})=\tilde{C}_{\rho}<0$, we have $\bar{u}\neq0$. Then, by a similar argument  as used in the  proof of Theorem \ref{Th41}, we can get $\bar{u}>0$ and $\bar{u}\in C^2(\Omega)$.

		\textbf{Step  2:} We prove  that $\tilde{C}_{\kappa}$ can be attained  by some $\tilde{u}\in \kappa$ with $\tilde{u}\in C^2(\Omega)$ and $\tilde{u}>0$.

		Due to step 1 and Lemma \ref{lemma52}, we obtain $-\infty<\tilde{C}_{\kappa}<0$.
		Take a minimizing sequence $\{u_n\}\subset\kappa$ for $\tilde{C}_{\kappa}$. Then $I'(u_n)=0$ and $I(u_n)\rightarrow\tilde{C}_{\kappa}$. By Lemma \ref{lemma23}, we can see that $\{u_n\}$ is bounded in $H_{0,r}^1(\Omega)$.~Hence, we can assume that
		\begin{align}
		\begin{array}{lll}
		u_n\rightharpoonup \tilde{u}~~\hbox{weakly~in}~H_{0,r}^1(\Omega),\\
		u_n\rightarrow \tilde{u}~~\hbox{strongly~in}~L^q(\Omega;|x|^\alpha),~1\leq q<2^*_{\alpha}, \\
		u_n\rightarrow \tilde{u}~~\hbox{strongly~in}~L^p(\Omega),~1\leq p<2^*, \\
		u_n\rightarrow \tilde{u}~~\hbox{a.e.}~\hbox{in}~\Omega.
		\end{array}
		\end{align}
		Then, going on as step 1, we obtain
		$$I(\tilde{u})=\tilde{C}_{\kappa}~~\hbox{and}~~I^\prime(\tilde{u})=0.$$
		
		Since $I(\tilde{u})=\tilde{C}_{\kappa}<0$, we have $\tilde{u}\neq0$. By a similar argument  as used in the proof Theorem \ref{Th41}, we can get $\tilde{u}>0$ and $\tilde{u}\in C^2(\Omega)$.

		\textbf{Step  3:} We will show that $\tilde{C}_\kappa=\tilde{C}_\rho$.

		Since $\bar{u}\in \kappa$, where $\bar{u}$ is the function  found in the step 1,  we have
		\begin{equation}\label{e6251}
		\tilde{C}_{\rho}=I(\bar{u})\geq\tilde{C}_\kappa.
		\end{equation}
		
		On the other hand, for the solution $\tilde{u}$ obtained in the step 2, we have that $I(\tilde{u})=\tilde{C}_{\kappa},~I'(\tilde{u})=0$ and $\tilde{u}>0$ in $\Omega$.~We consider the function $g(t):=I(t\tilde{u})$,~$t>0$. By
		Lemmas \ref{lemma22},  \ref{lemma53}  and the facts that $g(t)<0$ for $t>0$  small enough and $\lim\limits_{t\rightarrow+\infty}g(t)=-\infty$, we obtain  that $g(t)$ has only two extreme points $t_1, ~t_2\in(0,+\infty)$ with $t_1<t_2$. At the same
		time, we also have that $t_1$ and $t_2$ are the local minimum point and the maximum point of $g(t)$ respectively, and $g(t_2)>0$ and
		\begin{equation}\label{eq64}
		g(t_1)<g(t)<0~~\hbox{for any}~t\in(0,t_1).
		\end{equation}
		Since $I(\tilde{u})=\tilde{C}_{\kappa}<0$ and $I'(\tilde{u})=0$, we see that $g(1)<0$ and $g'(1)=0$, which tells us that
		1 is a extreme point of $g(t)$ and $g(1)<0$. So $t_1=1$.  According to \eqref{eq64}, we have that
		\begin{equation}\label{eq65}
		g(t)<0~~\hbox{for any}~t\in(0,1].
		\end{equation}
		
		We claim that $|\nabla\tilde{u}|_2<\rho$. In fact, if $|\nabla\tilde{u}|_2\geq\rho$, then  we can find a $t_3\in(0,1]$ such that $t_3|\nabla\tilde{u}|_2=\rho$, which, together with Lemma \ref{lemma22},  implies that $g(t_3)=I(t_3{\tilde{u}})\geq\delta>0$, contradicting to \eqref{eq65}. Thus, $|\nabla\tilde{u}|_2<\rho$, which implies that $\tilde{u}\in A$ and
		\begin{equation}\label{e6252}
		\tilde{C}_{\rho}\leq I(\tilde{u})=\tilde{C}_{\kappa}.
		\end{equation}

		According to \eqref{e6251} and \eqref{e6252}, we obtain that $\tilde{C}_\kappa=\tilde{C}_\rho$.
		
		We complete the proof.
	\end{proof}

	\section{Finding  a  Mountain pass solution for the case  of $\mu<0$}
	
	\begin{lemma}\label{lemma54}
		For any $\beta>0$, there exist  $B_1, B_2, B_3>0$ such that, for any $t\in(0,+\infty)$,
		\begin{equation*}
		g(\bar{u},tU_{\varepsilon,\alpha})\leq (tU_{\varepsilon,\alpha})^{2+\beta}+B_1(tU_{\varepsilon,\alpha})^2,~~x\in \Omega,
		\end{equation*}
		\begin{equation*}
		f(2^*_{\alpha},\bar{u},tU_{\varepsilon,\alpha})\geq \frac{2^*_{\alpha}}{2}L_1(tU_{\varepsilon,\alpha})^{2^*_{\alpha}-1}-B_2(tU_{\varepsilon,\alpha})^2,~~x\in \Omega, ~~2<N<6+2\alpha,
		\end{equation*}
		and
		\begin{equation*}
		|f(2^*_{\alpha},\bar{u},tU_{\varepsilon,\alpha})|\leq \frac{(2^*_{\alpha})^2}{2}L_2^{2^*_{\alpha}-2}(tU_{\varepsilon,\alpha})^2+B_2L_2(tU_{\varepsilon,\alpha})^{2^*_{\alpha}-1},~~x\in \Omega,~~N\geq3,
		\end{equation*}
		where $\bar{u}$ is the local minimum solution obtained in Theorem \ref{Th61}, $U_{\varepsilon,\alpha} $ is given in \eqref{eq31}, $$L_1:=\inf\limits_{x\in B_{2\rho}(0)}\bar{u},~~~L_2:=\sup\limits_{x\in B_{2\rho}(0)}\bar{u},$$
		$$
		g(x,y):=(x+y)^2\log(x+y)^2-x^2\log x^2-2xy\left(\log x^2+1\right),
		$$
		and
		$$
		f(2^*_{\alpha},x,y):=(x+y)^{2^*_{\alpha}}-x^{2^*_{\alpha}}-y^{2^*_{\alpha}}-2^*_{\alpha}x^{2^*_{\alpha}-1}y,
		$$
		
	\end{lemma}
	\begin{proof}
		Similar  proof  can be found in  \cite{heq}. So we omit it here.
	\end{proof}


	\begin{lemma}\label{lemma45}
		There exist  $\varepsilon_0>0$ and $T>4\rho/{S_{
				\alpha}^{\frac{N+\alpha}{2(\alpha+2)}}}$, dependent of $\varepsilon_0$, such that when $\varepsilon\leq \varepsilon_0$,
		\begin{equation*}
		I(\bar{u}+TU_{\varepsilon,\alpha})<I(\bar{u})<0.
		\end{equation*}
	\end{lemma}
	\begin{proof}
		By direct computations, we have that, for $\varepsilon$ small enough,
		\begin{align*}
		&I(\bar{u}+tU_{\varepsilon,\alpha})\\
		&=\frac{1}{2}\int|\nabla(\bar{u}+tU_{\varepsilon,\alpha})|^2-\frac{1}{2^*_{\alpha}}\int|x|^{\alpha}|\bar{u}
		+tU_{\varepsilon,\alpha}|^{2^*_{\alpha}}-\frac{\lambda}{2}\int|\bar{u}+tU_{\varepsilon,\alpha}|^2\\
		&-\frac{\mu}{2}\int(\bar{u}+tU_{\varepsilon,\alpha})^2\left(\log(\bar{u}+tU_{\varepsilon,\alpha})^2-1\right)\\
		&\leq \int|\nabla \bar{u}|^2+t^2\int|\nabla U_{\varepsilon,\alpha}|^2-\frac{t^{2^*_{\alpha}}}{2^*_{\alpha}}\int|x|^{\alpha}|U_{\varepsilon,\alpha}|^{2^*_{\alpha}}
		+C\int|\bar{u}+tU_{\varepsilon,\alpha}|^{2+\eta}+C\int|\bar{u}+tU_{\varepsilon,\alpha}|^{2-\eta}\\
		&\leq C+t^2\int|\nabla U_{\varepsilon,\alpha}|^2-\frac{t^{2^*_{\alpha}}}{2^*_{\alpha}}\int|x|^{\alpha}|U_{\varepsilon,\alpha}|^{2^*_{\alpha}}
		+Ct^{2+\eta}\int U_{\varepsilon,\alpha}^{2+\eta}+Ct^{2-\eta}\int U_{\varepsilon,\alpha}^{2-\eta}\\
		&\leq C+(2S_{\alpha}^{\frac{N+\alpha}{2+\alpha}})t^2-\frac{t^{2^*_{\alpha}}}{2^*_{\alpha}}(\frac{1}{2}S_{\alpha}^{\frac{N+\alpha}{2+\alpha}})
		+Ct^{2+\eta}+Ct^{2-\eta}\\
		&\rightarrow-\infty,~~\hbox{as}~t\rightarrow+\infty,
		\end{align*}
		where $0<\eta<\min\{2^*_{\alpha}-2,1\}$ and which implies that there exist $\varepsilon_0>0$ and $T>{4\rho}/{S_{
				\alpha}^{\frac{N+\alpha}{2(\alpha+2)}}}$, dependent of $\varepsilon_0$, such that when $\varepsilon\leq \varepsilon_0$,
		\begin{equation*}
		I(\bar{u}+TU_{\varepsilon,\alpha})<I(\bar{u})<0.
		\end{equation*}
	\end{proof}

	\begin{lemma}\label{lm741}
		For any $\varepsilon$ small  enough, there exists a $t_{\varepsilon}\in(0,T)$ such that
		\begin{equation*}
		I(\bar{u}+t_{\varepsilon}U_{\varepsilon,\alpha})=\max\limits_{t\in[0,T]}I(\bar{u}+tU_{\varepsilon,\alpha})\geq\delta
		\end{equation*}
		and
		\begin{equation*}
		0<\inf\limits_{\varepsilon}t_{\varepsilon}\leq t_{\varepsilon}\leq T.
		\end{equation*}
	\end{lemma}
	\begin{proof}
		We can prove it just as showing  Lemma 3.6 of   \cite{heq} and omit it here.
		%
	\end{proof}

	Set
	\begin{equation*}
	\gamma_{0}(t):=\bar{u}+tTU_{\varepsilon,\alpha},~t\in[0,1].
	\end{equation*}
	It is easy to see that $\gamma_{0}(t)\in \Gamma$ and $c_M\leq \sup\limits_{t\in[0,1]}I(\gamma_{0}(t))\leq I(\bar{u}+t_{\varepsilon}U_{\varepsilon,\alpha}).$

	\begin{lemma}\label{lemma56}
		Assume that $2<N<\min\{6,6+2\alpha\}$. Then  there exists $\varepsilon_1\leq\varepsilon_0$ such that when $\varepsilon\leq\varepsilon_1$,
		\begin{equation*}
		c_M\leq I(\bar{u}+t_{\varepsilon}U_{\varepsilon,\alpha})<\tilde{C}_{\kappa}+\frac{\alpha+2}{2(N+\alpha)}S_{\alpha}^{\frac{N+\alpha}{2+\alpha}}.
		\end{equation*}
	\end{lemma}
	\begin{proof}Since $2<N<\min\{6,6+2\alpha\}$, we can choose some $\beta>0$ such that $\min\{\frac{(2+\beta)(N-2)}{2},N-\frac{(2+\beta)(N-2)}{2}\}>\frac{N-2}{2}$ and $\min\{\alpha+2,N-2\}>\frac{N-2}{2}$.
		Following from Lemmas \ref{lemma54}, \ref{lm741} and Theorem \ref{Th61},   direct computations imply that, for $\varepsilon$ small enough,
		\begin{align*}
		&I(\bar{u}+t_{\varepsilon}U_{\varepsilon,\alpha})\\
		&=\frac{1}{2}\int|\nabla(\bar{u}+t_{\varepsilon}U_{\varepsilon,\alpha})|^2-\frac{1}{2^*_{\alpha}}\int|x|^{\alpha}|\bar{u}
		+t_{\varepsilon}U_{\varepsilon,\alpha}|^{2^*_{\alpha}}-\frac{\lambda}{2}\int|\bar{u}+t_{\varepsilon}U_{\varepsilon,\alpha}|^2\\
		&-\frac{\mu}{2}\int(\bar{u}+t_{\varepsilon}U_{\varepsilon,\alpha})^2\left(\log(\bar{u}+t_{\varepsilon}U_{\varepsilon,\alpha})^2-1\right)\\
		&=I(\bar{u})+\frac{1}{2}\int |\nabla U_{\varepsilon,\alpha}|^2t_{\varepsilon}^2-\frac{1}{2^*_{\alpha}}\int|x|^{\alpha}\left(|\bar{u}
		+t_{\varepsilon}U_{\varepsilon,\alpha}|^{2^*_{\alpha}}-\bar{u}^{2^*_{\alpha}}
		-2^*_{\alpha}\bar{u}^{2^*_{\alpha}-1}t_{\varepsilon}U_{\varepsilon,\alpha}\right)\\
		&-\frac{\lambda-\mu}{2}t_{\varepsilon}^2\int U_{\varepsilon,\alpha}^2-\frac{\mu}{2}\int\Big[(\bar{u}+t_{\varepsilon}U_{\varepsilon,\alpha})^2
		\log(\bar{u}+t_{\varepsilon}U_{\varepsilon,\alpha})^2-\bar{u}^2
		\log\bar{u}^2-2\bar{u}t_{\varepsilon}U_{\varepsilon,\alpha}\left(\log\bar{u}^2+1\right)\Big]\\
		&=I(\bar{u})+\frac{1}{2}\int |\nabla U_{\varepsilon,\alpha}|^2t_{\varepsilon}^2-\frac{1}{2^*_{\alpha}}\int|x|^{\alpha}
		|t_{\varepsilon}U_{\varepsilon,\alpha}|^{2^*_{\alpha}}
		-\frac{1}{2^*_{\alpha}}\int|x|^{\alpha}f\left(2^*_{\alpha},\bar{u},t_{\varepsilon}U_{\varepsilon,\alpha}\right)\\
		&-\frac{\lambda-\mu}{2}t_{\varepsilon}^2\int U_{\varepsilon,\alpha}^2-\frac{\mu}{2}\int g\left(\bar{u},t_{\varepsilon}U_{\varepsilon,\alpha}\right)\\
		&\leq I(\bar{u})+\frac{1}{2}\int |\nabla U_{\varepsilon,\alpha}|^2t_{\varepsilon}^2-\frac{t_{\varepsilon}^{2^*_{\alpha}}}{2^*_{\alpha}}
		\int|x|^{\alpha}|U_{\varepsilon,\alpha}|^{2^*_{\alpha}}-\frac{L_1}{2}t_{\varepsilon}^{2^*_{\alpha}-1}
		\int|x|^{\alpha}|U_{\varepsilon,\alpha}|^{2^*_{\alpha}-1}\\
		&+\frac{B_2}{2^*_{\alpha}}t_{\varepsilon}^2\int|x|^{\alpha}|U_{\varepsilon,\alpha}|^{2}-\frac{\lambda-\mu}{2}t_{\varepsilon}^2\int U_{\varepsilon,\alpha}^2-\frac{\mu}{2}t_{\varepsilon}^{2+\beta}\int|U_{\varepsilon,\alpha}|^{2+\beta}-\frac{\mu B_1}{2}t_{\varepsilon}^2\int U_{\varepsilon,\alpha}^2\\
		&\leq\tilde{C}_{\rho}+\frac{\alpha+2}{2(N+\alpha)}S_{\alpha}^{\frac{N+\alpha}{2+\alpha}}
		+O(\varepsilon^{N-2})-C\varepsilon^{\frac{N-2}{2}}\\
		&+O(\varepsilon^{\min\{\alpha+2,N-2\}}\log\frac{1}{\varepsilon})+C\int U_{\varepsilon,\alpha}^2+O(\varepsilon^{\min\{\frac{(2+\beta)(N-2)}{2},N-\frac{(2+\beta)(N-2)}{2}\}}\log\frac{1}{\varepsilon})\\
		&<\tilde{C}_{\kappa}+\frac{\alpha+2}{2(N+\alpha)}S_{\alpha}^{\frac{N+\alpha}{2+\alpha}},
		\end{align*}
		where  we have used  the fact  that
		\begin{gather*}
		\int_{\Omega}\left|U_{\varepsilon,\alpha}\right|^2=
		\begin{cases}
		O(\varepsilon),~&\hbox{if}~N=3,\\
		d\varepsilon^2\log\frac{1}{\varepsilon}+O(\varepsilon^2),~&\hbox{if}~ N=4,\\ d\varepsilon^2+O(\varepsilon^{N-2}),~&\hbox{if}~ N\ge5.
		\end{cases}
		\end{gather*}

		We complete the proof.
	\end{proof}

	\begin{proof}[\bf Proof of Theorem \ref{Th64}:]
		It follows from the Mountain pass Theorem that there exists
		a sequence $\{u_n\}\subset H_{0,r}^1(\Omega)$  such that, as  $n \rightarrow \infty$,
		
		$$I\left(u_{n}\right) \rightarrow c_{M},\quad I^{\prime}\left(u_{n}\right) \rightarrow 0,\quad \text { in }\big(H_{0,r}^{1}(\Omega)\big)^{-1}.$$
		According to Lemma \ref{lemma25} and Lemma \ref{lemma56}, we can see that
		there exists $u\in H_{0,r}^1(\Omega)$ such that
		$$u_n\rightarrow u~~\hbox{strongly in}~H_{0,r}^1(\Omega),$$
		which implies that $I(u)=c_M$ and $I'(u)=0$. That is, $u$ is a Mountain pass solution of \eqref{eq11}. Then, by a similar argument  as used in the  proof of Theorem \ref{Th41}, we can get $u>0$ and $u\in C^2(\Omega)$.
	\end{proof}

	\section{The proof of the non-existence of positive solutions}

	\begin{proof}[\bf Proof of Theorem \ref{Th16}:]
		We argue by contradiction.
		Assume that problem $\eqref{eq11}$ has a positive solution $u_0$, and let $\varphi_1(x)>0$ be the first eigenfunction corresponding to $\lambda_1(\Omega)$. Then
		\begin{align*}
		&\int_\Omega\Big(|x|^{\alpha}u_0^{2^*_{\alpha}-1}+\lambda u_0+\mu u_0\log u_0^2\Big)\varphi_1(x)\\
		&=\int_\Omega-\Delta u_0\varphi_1(x)=\int_\Omega-\Delta\varphi_1(x)u_0=\int_\Omega\lambda_1(\Omega)\varphi_1(x)u_0,
		\end{align*}
		which implies that
		\begin{equation}\label{eq115}
		\int_\Omega\Big(|x|^{\alpha}u_0^{2^*_{\alpha}-2}+\lambda-\lambda_1(\Omega)+\mu\log u_0^2\Big)u_0\varphi_1(x)=0.
		\end{equation}
		Define
		\begin{equation*}
		f(s):=s^{2^*_{\alpha}-2}+\lambda-\lambda_1(\Omega)+\mu\log s^2,~~s>0.
		\end{equation*}
		Then
		\begin{equation*}
		f'(s)=(2^*_{\alpha}-2)s^{2^*_{\alpha}-3}+2\mu\frac{1}{s}.
		\end{equation*}
		By a direct computation, $f'(s)=0$ has a unique root  $s_0=\big(-\frac{(N-2)\mu}{\alpha+2}\big)^{\frac{N-2}{2(\alpha+2)}}$. Furthermore, $f'(s)<0$ in $\big(0,\big(-\frac{(N-2)\mu}{\alpha+2}\big)^{\frac{N-2}{2(\alpha+2)}}\big)$ and $f'(s)>0$ in $\big(\big(-\frac{(N-2)\mu}{\alpha+2}\big)^{\frac{N-2}{2(\alpha+2)}},+\infty\big)$.
		
		When $\alpha\in(-2,0]$ and $|x|\leq1$, we have that, for any $s>0$,
		\begin{equation}\label{eq116}
		\begin{split}
		&|x|^{\alpha}s^{2^*_{\alpha}-2}+\lambda-\lambda_1(\Omega)+\mu\log s^2\geq f(s)
		\geq f\Big(\Big(-\frac{(N-2)\mu}{\alpha+2}\Big)^{\frac{N-2}{2(\alpha+2)}}\Big)\\
		&=-\frac{(N-2)\mu}{\alpha+2}+\frac{(N-2)\mu}{\alpha+2}\log\Big(-\frac{(N-2)\mu}{\alpha+2}\Big)+\lambda-\lambda_1{(\Omega)}\geq0.
		\end{split}
		\end{equation}
		Since  $u_0>0$ and $\varphi_1>0$, we have that
		\begin{equation}\label{eq741}
		\int_\Omega\Big(|x|^{\alpha}u_0^{2^*_{\alpha}-2}+\lambda-\lambda_1(\Omega)+\mu\log u_0^2\Big)u_0\varphi_1(x)\geq\int_\Omega f(u_0(x))u_0\varphi_1(x)>0.
		\end{equation}
		Otherwise, $|x|^{\alpha}u_0^{2^*_{\alpha}-2}+\lambda-\lambda_1(\Omega)+\mu\log u_0^2=0$ a.e.~in $\Omega$,~i.e., $f(u_0(x))=0$ a.e.~in $\Omega$. So $u_0(x)=\big(-\frac{(N-2)\mu}{\alpha+2}\big)^{\frac{N-2}{2(\alpha+2)}}$ a.e.~in $\Omega$, which contradicts to $u_0\in H_{0,r}^1(\Omega)$. By \eqref{eq115} and  \eqref{eq741}, we have that
		\begin{equation*}
		0=\int_\Omega\left(|x|^{\alpha}u_0^{2^*_{\alpha}-2}+\lambda-\lambda_1(\Omega)+\mu\log u_0^2\right)u_0\varphi_1(x)>0,
		\end{equation*}
		which is a contradiction. Hence  problem \eqref{eq11} has no positive solutions.
		
		This completes the proof.
	\end{proof}

	\section*{Acknowledgement}
	The authors   would like to thank the anonymous referees for carefully reading this paper and  making valuable comments and suggestions.
	This work  was  supported by the fund from NSF of China (No.  12061012,12461022).
	
	
	\textbf{Data availability} ~~ The manuscript has no associated data, therefore can be considered that all data needed are
	available freely.

	\textbf{Conflict of interest}~~ The authors declare that they have no Conflict of interest.

	\textbf{Ethics approval}~~ The research does not involve humans and/or animals. The authors declare that there are no
	ethics issues to be approved or disclosed.

	%
	%
	%

	\bibliographystyle{plainnat}
	
\end{document}